\documentclass{article}
\usepackage{amsmath}
\usepackage{amsthm}
\usepackage{amssymb}
\usepackage[dvips]{graphicx}
\usepackage{latexsym}
\usepackage{mathtools}
\usepackage{geometry}
 \renewcommand{\equation}

\newtheorem{prop}{Proposition}[section]

\newtheorem{thm}{Theorem}[section]
\newtheorem{cor}{Corollary}[section]
\newtheorem{fact}{Fact}[section]
\newtheorem{corresp*}{Correspondence}[section]
\newtheorem{facts*}{Facts}[section]

\theoremstyle{definition}
\newtheorem{rmk}{Remark}[section]
\newtheorem{defini}{Definition}[section]

\title{The existence of homologically fibered links and solutions of some equations}
\author{Nozomu Sekino}
\date{}

\begin{document}
\maketitle

\begin{abstract}
There is one generalization of fibered links in 3-manifolds, called homologically fibered links. 
It is known that the existence of a homologically fibered link whose fiber surface has a given homeomorphic type is determined by the first homology group and its torsion linking form of the ambient 3-manifold. 
In this paper, we interpret the existence of homologically fibered links with that of a solution of some equation, in terms of the first homology group and its torsion linking form or a surgery diagram of the ambient manifold. 
As an application, we compute the invariant ${\rm hc}(\cdot)$, defined through homologically fibered knots, for 3-manifolds whose torsion linkng forms represent a generator of linkings. 
\end{abstract}

\section{Introduction}
It is known that every connected orientable closed 3-manifold $M$ has a fibered link \cite{alexander}, i.e., a link $L$ in $M$ such that the complement of a small open neighborhood of $L$ admits a structure of a connected orientable compact surface bundle over $S^1$ and that each boundary component of a fiber surface is a longitude of a component of $L$. 
Moreover, it is known that every connected orientable closed 3-manifold $M$ has a fibered knot \cite{myers}. 
However, finding fibered links in a given 3-manifold is difficult in general. For example ${\rm op}(M)$, defined as the minimal genus of fiber surfaces of fibered knots which $M$ has, is a topological invariant which is difficult to calculate. 
There is one generalization of fibered links, called {\it homologically fibered links} \cite{goda_sakasai} (defined in Definition~\ref{homologycob} below).
A homologically fibered link requests some Seifert surface such that the result of cut a 3-manifold along the surface is a homological product of a surface and an interval,
 whereas a fibered link requests some Seifert surface such that the result of cutting a 3-manifold along the surface is the product of a surface and an interval. 
Clearly, a fibered link is also a homologically fibered link,
 and ${\rm hc}(M)$ \cite{goda_sakasai}, defined as the minimal genus of homological fiber surfaces of homologically fibered knots which $M$ has gives a lower bound of ${\rm op}(M)$. 
Homological products of surfaces (of a fixed homeomorphic type) and intervals are not merely the homological constraint for fiber structures, but they form a monoid by stacking, which contains the mapping class group as a submonoid (see \cite{goda_sakasai}). 
This attracts attention as a generalization of the mapping class groups. 
Homological products of surfaces and intervals are studied using ``clasper theory''. 
As a consequence of clasper theory,  it is known that the existence of a homologically fibered link whose fiber is a given homeomorphic type in $M$ depends on the first homology group of $M$ with its torsion linking form. 
In this paper, we interpret the existence of a homologically fibered link whose fiber is a given homeomorphic type with the existence of a solution for some equation as follows. \\

As a convention, for an $m\times m$-matrix $A$ and an $n\times n$-matrix $B$, $A\oplus B$ is an $(m+n)\times(m+n)$-matrix whose $(i,j)$-entry is the same as the $(i,j)$-entry of $A$ for $1\leq i,j\leq m$, the $(m+i',m+j')$-entry is the same as the $(i',j')$-entry of $B$ for $1\leq i',j'\leq n$, and the other entries are $0$.\\
Let $O_{r}$ be an $r\times r$-zero-matrix, 
$I_{r}$ an $r\times r$-identity matrix, and we regard $0\times 0$-matrix ($O_{0}$ and $I_{0}$) as the identity element for $\oplus$. \\
Let $P$, $Q$, $F^{k}_{0}$, $F^{k}_{1}$ and $G^{k}_{1}$ be matrices as follows, where $a\geq 0$ and $k\geq 0$:
\begin{itemize}
 \item $P$ is a diagonal $a\times a$-matrix whose $(i,i)$-entry is $p_{i}\in \mathbb{Z}$ for $a>0$ or $I_{0}$ for $a=0$.
 \item $Q$ is a diagonal $a\times a$-matrix whose $(i,i)$-entry is $-q_{i}\in \mathbb{Z}$ for $a>0$ or $I_{0}$ for $a=0$.
 \item $F^{0}_{0}$, $F^{0}_{1}$ and $G^{0}_{1}$ to be $I_{0}$,\\ and
$F^{k}_{0}=
   \left( \begin{array}{cc}
   0   & 2^{k}   \\
   2^{k} & 0 \\
  \end{array} \right)   $ , \hspace{1.0cm}
           $F^{k}_{1}=
   \left( \begin{array}{cc}
   2^{k+1}   & -2^{k}   \\
   -3\cdot2^{k} & 2^{k+1} \\
  \end{array} \right)$ , \hspace{1.0cm}
           $G^{k}_{1}=
   \left( \begin{array}{cc}
   -1 &  0      \\
  0    & -3   \\
  \end{array} \right)$  for $k>0$.

\end{itemize}

Also let $S$ and $T$ be matrices as follows, where $r\geq0$, $e_{0}\geq 0$, $e_{1}\geq 0$, $k_{0}=k'_{0}=0$, $k_{i}\geq 1$ for $i\neq0$, and $k'_{j}\geq 2$ for $j\neq0$:
\begin{itemize}
 \item $S=O_{r}\oplus P\oplus (\bigoplus_{0\leq i\leq e_{0}}F^{k_{i}}_{0}) \oplus (\bigoplus_{0\leq j\leq e_{1}}F^{k'_{j}}_{1})$. 
 \item $T=I_{r}\oplus Q\oplus I_{2e_{0}} \oplus (\bigoplus_{0\leq j\leq e_{1}}G^{k'_{j}}_{1})$. 
\end{itemize}

As usual notation, $\Sigma_{g,n+1}$ denotes a connected orientable compact surface of genus $g$ with $n+1$ boundary components. 

\begin{thm} \label{thm}
Let $M$ be a connected closed oriented 3-manifold whose free part of the first homology group with integer coefficient is isomorphic to $\mathbb{Z}^{r}$ and whose torsion linking form is equivalent to $(\bigoplus_{0\leq l\leq a}A^{p_l}(q_l))\oplus (\bigoplus_{0\leq i \leq e_0} E^{k_i}_{0}) \oplus (\bigoplus_{0\leq j\leq e_1}E^{k'_{j}}_1)$, where $a,e_0,e_1$ are non-negative integers and notations $A^{p_l}(q_l), E^{k_i}_{0}, E^{k'_{j}}_{1}$ represent linkings, defined in Subsection~\ref{linkingpair}. \\
Suppose $(g,n)\neq (0,0)$. 
Then $M$ has a homologically fibered link whose homological fiber is homeomorphic to $\Sigma_{g,n+1}$ if and only if there exist $(r+a+2e_{0}+2e_{1}) \times (2g+n)$-matrix of integer coefficients $X$ and $(2g+n) \times (2g+n)$- symmetric matrix of integer coefficients $Y$ satisfying the following equation:
\begin{eqnarray} \label{condition}
   \left| \begin{array}{cc}
     S    &  TX  \\
      X^{t}    &  Y+(\mathcal{E}\oplus O_{n}) \\
  \end{array} \right| = \pm 1,
\end{eqnarray}
where $\mathcal{E}$ is a $2g\times 2g$-matrix whose $(2i-1,2i)$-entry is $1$ for every $1\leq i\leq g$ and the others are $0$.
\end{thm}

\begin{rmk}
For the case which is not covered by Theorem~\ref{thm},
 it is known that a connected closed orientable 3-manifold $M$ has a homologically fibered link whose homological fiber is homeomorphic to $\Sigma_{0,1}$ if and only if $M$ is an integral homology $3$-sphere. See the definitions in the Subsection~\ref{defs}.
\end{rmk}

In general, it is not easy to compute the torsion linking form of a 3-manifold. 
Sometimes surgery diagrams of a 3-manifold may be easy to handle. 
A statement in terms of surgery diagrams is given in Proposition~\ref{condition_surgery}.  

The rest of this paper is organized as follows: 
In Section~\ref{sec_homfiblinks}, we recall the definition of homologically fibered links and the fact of their dependence on torsion linking forms. 
In Section~\ref{sec_reps}, we recall three types of generators of linking forms and give rational homology 3-spheres whose torsion linking forms are the generators. 
In Section~\ref{surf_bands}, we give a correspondence between surfaces in 3-manifolds and tuples of annuli and bands for Theorem~\ref{thm}.  
In Section~\ref{proof}, we give a proof of Theorem \ref{thm}. 
In Section~\ref{sec_examples}, we compute ${\rm hc}(\cdot)$ for 3-manifolds whose torsion linking forms are two of three types of generators in Section~\ref{sec_reps}.

\subsection*{Acknowledgements}
The author would like to thank Takuya Sakasai and Yuta Nozaki for introducing him this subject and giving him many ingredients about it. 
He also is grateful to referee for his or her kindness, patience and correcting many mistakes in the proofs and the arguments.

\section{Homologically fibered links} \label{sec_homfiblinks}
In this section, we review the definition of homologically fibered links and the fact that the existence of homologically fibered links depends on the torsion linking forms. 
In this paper, the homology groups are with integer coefficients. 
\subsection{Definitions}\label{defs}
\begin{defini} (homology cobordism, {\cite[Section~2.4]{garoufalidis}})\\
A {\it homology cobordism over} $\Sigma_{g,n+1}$ is a triad ($X,\partial_{+}X,\partial_{-}X$), 
where $X$ is an oriented connected compact 3-manifold and $\partial_{+}X \cup \partial_{-}X$ is a partition of $\partial X$, and $\partial_{\pm} X$ are homeomorphic to $\Sigma_{g,n+1}$ satisfying:
\begin{itemize}
  \item $\partial_{+}X \cup \partial_{-}X =  \partial X$.
  \item $\partial_{+}X \cap \partial_{-}X =  \partial(\partial_{+}X)$.
  \item The induced maps $(i_{\pm})_{*}: H_{*}(\partial_{\pm} X) \rightarrow H_{*}(X)$ are isomorphisms, where $i_{\pm}: \partial_{\pm}X \rightarrow X$ are the inclusions.\\
\end{itemize}
\end{defini}

Note that the third condition is equivalent to the condition that $i_{\pm}$ induce isomorphisms on $H_{1}(\cdot)$. 
Moreover by using the Poincar$\rm{\acute{e}}$-Lefschetz duality, we see that it is sufficient to require only one of $i_{+}$ and $i_{-}$ to induce an isomorphism on $H_{1}(\cdot)$. 

\begin{defini}\label{homologycob} \cite{goda_sakasai}
Let $L$ be a link in a closed orientable 3-manifold $M$. 
$L$ is called {\it homologically fibered link} if there exists a Seifert surface $F$ of $L$ such that $(M\setminus F, F_{+}, F_{-})$ is a homology cobordism over a surface homeomorphic to $F$, where $F_{\pm}$ are the cut ends. 
In this situation, we call $F$ a {\it homological fiber} of $L$.
\end{defini}

Clearly every fibered link is a homologically fibered link since the former requests that the manifold cut opened by some Seifert surface to be the product of an interval and a surface and the latter requests it to be the homological product of an interval and a surface. 
By the observations above the Definition~\ref{homologycob}, for a Seifert surface $F$ in a closed 3-manifold $M$ to be a homological fiber, it is enough to check that push-ups (or push-downs) of oriented simple closed curves on $F$ which form a basis of $H_{1}(F)$ also form a free basis of $H_{1}(M\setminus Int(F\times[0,1]))$. 

\begin{rmk}\label{lowerbound}
By definition, when a connected closed orientable 3-manifold $M$ has a homologically fibered link with homological fiber $F$, $H_{1}(M)$ can be generated by $1-\chi(F)$ elements, where $\chi(F)$ is the Euler number of $F$. 
Therefore the minimal order of generating sets for $H_{1}(M)$ imposes some constraint on homeomorphic types of homological fibers in $M$.
\end{rmk}

\subsection{Dependence on the torsion linking forms}

\begin{defini}(Torsion linking form)\\
Let $M$ be a connected closed oriented 3-manifold, $TH_{1}(M)$ the torsion part of $H_{1}(M)$. 
For every $a\in TH_{1}(M)$, there is a non-zero integer $n$ such that $na$ vanishes in $H_{1}(M)$, and we fix such an integer $n_a$ of minimal absolute value. 
This $a$ has a representative as oriented curves $L_a$ in $M$ such that the $n_a$-parallel copy of $L_a$ bounds a surface $S_{a}$ in $M$. 
The {\it torsion linking form} of $M$, $\psi_{M} : TH_{1}(M)\times TH_{1}(M) \longrightarrow \mathbb{Q}/\mathbb{Z}$ maps $(a,b)$ to $\frac{<S_{a}, L_{b}>}{n_a}\ {\rm mod} \ \mathbb{Z} $, where $<S_{a}, L_{b}>$ represents the algebraic intersection number of $S_a$ and $L_b$. 
It is known that $\psi_{M}$ is well-defined and is non-singular, symmetric bilinear pairing, and that $\psi_{-M}=-\psi_{M}$, where $-M$ denotes the manifold homeomorphic to $M$ with the opposite orientation.
\end{defini}

There is the operation of 3-manifolds preserving the first homology groups and the torsion linking forms, called the ``Borromean surgery'' introduced in \cite{matveev}. 

\begin{defini}
Let $V$ be a standard handlebody of genus three in $S^3$ which contains $0$-framed link as in Figure \ref{borromeansurg}. 
Consider a compact orientable 3-manifold $M$ and an embedding $f : V\longrightarrow M$. 
Let $N$ be a manifold obtained as a result of the Dehn surgery on $M$ along the framed link in $f(V)$. 
This $N$ is called the manifold obtained from $M$ (and $f$) by a {\it Borromean surgery}. 
\end{defini}

\begin{figure}[htbp]
 \begin{center}
  \includegraphics[width=30mm]{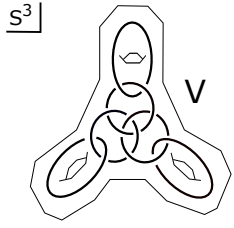}
 \end{center}
 \caption{A standard handlebody of genus three containing a $0$-framed link}
 \label{borromeansurg}
\end{figure}

About Borromean surgeries, the following facts are known:

\begin{fact}\label{borrom1}{\rm \cite{habiro}, \cite{massuyeau}}
Let $X$ be a homology cobordism over $\Sigma_{g,n+1}$ and $Y$ a manifold obtained from $X$ by a Borromean surgery. 
Then $Y$ is also a homology cobordism over $\Sigma_{g,n+1}$.
\end{fact}

\begin{fact} \label{borrom2} {\rm \cite{matveev}}
Let $M$ and $N$ be closed oriented 3-manifolds. 
Then the following are equivalent:
\begin{itemize}
\item There is an isomorphism between $H_1(M)$ and $H_1(N)$ such that $\psi_{M}$ and $\psi_{N}$ are equivalent under this isomorphism.
\item $M$ is obtained from $N$ by a finite sequence of Borromean surgeries.
\end{itemize}
\end{fact}

\vspace{0.5cm}
By the facts above, we can see the dependence on the first homology groups and the torsion linking forms about the existence of homologically fibered links: 
\begin{prop}\label{dependence}
Suppose that for connected closed oriented 3-manifolds $M$ and $N$, there exists an isomorphism between their first homology groups such that their torsion linking forms are equivalent under the isomorphism. 
Then $N$ has a homologically fibered link whose homological fiber is homeomorphic to $\Sigma_{g,n+1}$ if and only if $M$ has a homologically fibered link whose homological fiber is homeomorphic to $\Sigma_{g,n+1}$.
\end{prop}
\begin{proof}
We prove the if part and note that it is enough. 
By Fact~\ref{borrom2}, there is a finite sequence of Borromean surgeries starting from $M$ and ending at $N$. 
Since it is enough to prove when $N$ is obtained from $M$ by one Borromean surgery,
 we assume that. 
Suppose $M$ has a homological fibered link with homological fiber $F$ which is homeomorphic to $\Sigma_{g,n+1}$. 
Take a spine $T$ of $F$ and assume that $F$ lies in a small neighborhood of $T$. 
Isotope $F$ so that it is disjoint from the surgery link of the Borromean surgery. 
This can be done since $T$ and the surgery link are graphs in a 3-manifold. 
Let $F'$ be a surface, which is the image of $F$ after the Borromean surgery. 
Since $F$ is disjoint from the surgery link, a triad $(N\setminus F', F'^{+},F'^{-})$ is obtained from $(M\setminus F, F^{+},F^{-})$ by the Borromean surgery. 
Then by Fact~\ref{borrom1}, $(N\setminus F', F'^{+},F'^{-})$ is a homology cobordism over $\Sigma_{g,n+1}$. 
This implies that $N$ has a homologically fibered link with homological fiber $F'$ which is homeomorphic to $\Sigma_{g,n+1}$. 
\end{proof}

Note that if a connected closed oriented 3-manifold $M$ has a homologically fibered link whose homological fiber is homeomorphic to $\Sigma_{g,n+1}$, then $-M$ also has such a homologically fibered link. 
By Proposition~\ref{dependence}, for a connected closed oriented 3-manifold $M$ containing a homologically fibered link with  homological fiber of a given homeomorphic type we can replace $M$ with another one whose first homology group with the torsion linking form is the same as arbitrary one of $M$ and $-M$.

\section{Representatives with respect to the first homology groups and their torsion linking forms} \label{sec_reps}
In this section, at first we recall a part of the result of \cite{wall}, which gives a generator of the semigroup of the linkings on finite abelian groups, and next 
we fix representatives of 3-manifolds with respect to the first homology groups and their torsion linking forms. 
In fact, representatives are given in \cite{kawauchi_kojima}. 
We give another representatives for our use. Note that they may coincide. 

\subsection{Linking pair}\label{linkingpair}
A {\it linking} is a pair $(G,\psi)$ such that $G$ is a finite abelian group and $\psi$ is a non-singular, symmetric bilinear pairing $G\times G\longrightarrow \mathbb{Q}/\mathbb{Z}$. 
Sometimes $\psi$ is called a linking on $G$. 
We call two linkings $(G,\psi)$ and $(G',\psi')$ equivalent if there exists a group isomorphism between $G$ and $G'$ through which $\psi$ and $\psi'$ are equivalent. 
By fixing a basis of $G$, $\psi$ is represented as a non-singular, symmetric matrix with coefficients in $\mathbb{Q}/\mathbb{Z}$, whose $(i,j)$-entry is the image of the $i$-th element and the $j$-th element of the basis under $\psi$. 
For two likings $(G_1,\psi_1)$ and $(G_2,\psi_2)$, define the product as $(G_1\oplus G_2, \psi_1\oplus \psi_2)$, where $\psi_1\oplus \psi_2$ maps $\left( (g_1,g_2),(h_1,h_2)\right)$ to $\psi_1(g_1,h_1)+\psi_2(g_2,h_2)$ for every $g_{1}, h_{1}\in G_{1}$ and $g_{2}, h_{2}\in G_{2}$. 
Under this product, linkings form an abelian semigroup $\mathcal{R}$. 
For a connected closed orientable 3-manifold $M$, its torsion linking form, $(TH_1(M), \psi_{M})$ is an example of linkings. 
Moreover, for two connected closed oriented 3-manifolds $M$ and $N$, the torsion linking form of $M\# N$ is $\psi_{M}\oplus \psi_{N}$. 
In \cite{wall}, a generator of $\mathcal{R}$ is given as following:

\begin{fact}\label{gen_of_r} {\rm \cite{wall}}
\begin{itemize}
\item Let $A^{p}(q)$ be a linking on $\mathbb{Z}/p\mathbb{Z}$ (with a generator $1\in \mathbb{Z}/p\mathbb{Z}$) which is represented by 
   $\left( \begin{array}{c}
      \frac{q}{p}  \\
    \end{array} \right)$
for coprime integers $p>1$ and $q$, and

\item let $E^{k}_{0}$ be a linking on $\mathbb{Z}/2^{k}\mathbb{Z} \oplus \mathbb{Z}/2^{k}\mathbb{Z}$ (with a basis $\{(1,0), (0,1)\}$) which is represented by 
   $\left( \begin{array}{cc}
   0       &  2^{-k}  \\
   2^{-k}  &    0   \\
  \end{array} \right)$
for every integer $k>0$, and

\item let $E^{k}_{1}$ be a linking on $\mathbb{Z}/2^{k}\mathbb{Z} \oplus \mathbb{Z}/2^{k}\mathbb{Z}$ (with a basis $\{(1,0), (0,1)\}$) which is represented by 
   $\left( \begin{array}{cc}
   2^{1-k}       &  2^{-k}  \\
   2^{-k}  &    2^{1-k}   \\
  \end{array} \right)$
for every integer $k>1$.
\end{itemize}
Then $\mathcal{R}$ is generated by $A^{p}(q)$, $E^{k}_{0}$ and $E^{k}_{1}$.
\end{fact}

Moreover, the presentation of $\mathcal{R}$ is revealed by combining the results in \cite{wall} and \cite{kawauchi_kojima}. 
As a notation we regard $A^{1}(q)$, $E^{0}_{0}$ and $E^{0}_{1}$ as $\phi$, the identity element of $\mathcal{R}$.\\
Thus every linking is represented as $(\oplus_{0\leq l\leq a}A^{p_l}(q_l))\oplus (\oplus_{0\leq i \leq e_0} E^{k_i}_{0}) \oplus (\oplus_{0\leq j\leq e_1}E^{k'_{j}}_1)$, where $a\geq 0$, $p_{0}=1$, $p_{l}$ is an integer greater than $1$ for $l\neq 0$, $q_{l}$ is an integer prime to $p_{l}$, and $e_{0}\geq 0$, $k_{0}=0$, $k_{i}>0$ for $i\neq 0$, and $e_{1}\geq 0$, $k'_{0}=0$, $k'_{j}>1$ for $j\neq 0$. Note that this representation is not unique.

\subsection{Representatives}
We give 3-manifolds representing the generators in Fact~\ref{gen_of_r}. 
On showing that the torsion linking forms of these are the generators, we adopt a method of calculating torsion linking forms of rational homology 3-spheres by using their Heegaard splittings \cite{conway}. 
At first, we review the method.

\subsubsection{Calculating torsion linking forms by using Heegaard splittings \cite{conway}} \label{cal_tor}
Let $M$ be a rational homology 3-sphere and $M=V\cup W$ a Heegaard splitting i.e. $M$ is obtained from two handlebodies of same genus, say $g$, $V$ and $W$ by pasting their boundaries by some orientation reversing homeomorphism, say $f: \partial V\longrightarrow \partial W$. 
We give $V$ and $W$ orientations as standard handlebodies in $\mathbb{R}^{3}$, and we give $M$ an orientation coming from $V$. 
Take a symplectic basis $\{a^{V}_1, \dots a^{V}_g, b^{V}_1, \dots, b^{V}_g\}$ of $H_1(\partial V)$ such that $b^{V}_i$ is $0$ in $H_1(V)$ for every $i$, where symplectic basis means that the intersection form on $\partial V$ satisfies $\langle a^{V}_i, a^{V}_j \rangle=0$, $\langle b^{V}_i,b^{V}_j \rangle=0$, and $\langle a^{V}_i,b^{V}_j \rangle =\delta_{i,j}$ for every $1\leq i,j\leq g$. 
Similarly, take a symplectic basis $\{a^{W}_1, \dots a^{W}_g, b^{W}_1, \dots, b^{W}_g\}$ of $H_1(\partial W)$ such that $b^{W}_i$ is $0$ in $H_1(W)$ for every $i$. 
Denote by 
   $\left( \begin{array}{cc}
   A       & B  \\
   C  &    D   \\
  \end{array} \right)$
the matrix representing $f_{*}$, the map between the first homology group induced by $f$ with respect to these bases, where $A$, $B$, $C$ and $D$ are $(g\times g)$-matrices over $\mathbb{Z}$. 

\begin{fact} \label{cfh} {\rm \cite{conway}}
In the situation above, $det(B)\neq 0$ and $H_1(M)$ is isomorphic to $\mathbb{Z}^{g}/B^{t}\mathbb{Z}^{g}$. 
Moreover, the torsion linking form $\psi_M$ is equivalent to $\mathbb{Z}^{g}/B^{t}\mathbb{Z}^{g} \times \mathbb{Z}^{g}/B^{t}\mathbb{Z}^{g} \longrightarrow \mathbb{Q}/\mathbb{Z}\ ; (v,w) \mapsto -v^{t}\left( B^{-1}A\right)w$. 
\end{fact}

\begin{rmk}
We review the procedure for the calculation in terms of Heegaard diagrams rather than gluing maps used in the following. 
Let $M=V\cup W$ be a genus $g$ Heegaard splitting of a closed oriented 3-manifold (not necessarily a rational homology three sphere), and $S$ the splitting surface. 
Give $S$ the orientation as $\partial V$. 
Take a family of pairwise disjoint oriented simple closed curves $\{b^{V}_{1},\dots,b^{V}_{g}\}$ on $S$ such that each $b^{V}_{i}$ bounds pairwise disjoint disks in $V$ and that these disks cut $V$ into a ball. 
Similarly, take a family of pairwise disjoint oriented simple closed curves $\{b^{W}_{1},\dots,b^{W}_{g}\}$ on $S$ such that each $b^{W}_{i}$ bounds pairwise disjoint disks in $W$ and that these disks cut $W$ into a ball. 
Note that the triplet $(S, \{b^{V}_{1},\dots,b^{V}_{g}\}, \{b^{W}_{1},\dots,b^{W}_{g}\})$ determines the homeomorphic type of $M=V\cup W$. 
It is well-known that we can compute $H_{1}(M)$ using the curves: $H_{1}(M)=\mathbb{Z}\langle x_{1},\dots ,x_{g}\rangle /  \left( \Sigma^{g}_{i=1} \langle b^{V}_{i}, b^{W}_{j}\rangle_{S} \cdot x_{i} \ {\rm for }\ 1\leq j \leq g  \right) $, where $\langle \cdot, \cdot \rangle_{S}$ is the algebraic intersection form on $S$. 
By this, we can see whether $M$ is a rational homology $3$-sphere or not. 
Suppose that $M$  is a rational homology $3$-sphere in the following. 
Choose a family of pairwise disjoint oriented simple closed curves $\{a^{V}_{1},\dots,a^{V}_{g}\}$ on $S$ such that $\langle a^{V}_i,b^{V}_j \rangle_{S} =\delta_{i,j}$ for every $1\leq i,j\leq g$. 
Similarly, choose a family of pairwise disjoint oriented simple closed curves $\{a^{W}_{1},\dots,a^{W}_{g}\}$ on $S$ such that $\langle a^{W}_i,b^{W}_j \rangle_{S} =-\delta_{i,j}$ for every $1\leq i,j\leq g$. 
Then the homology class of $\{a^{V}_{1},\dots,a^{V}_{g},b^{V}_{1},\dots,b^{V}_{g}\}$ in $S$ forms a symplectic basis of $H_{1}(\partial V)$ such that each $b^{V}_{i}$ vanishes in $H_{1}(V)$
, and the homology class of $\{a^{W}_{1},\dots,a^{W}_{g},b^{W}_{1},\dots,b^{W}_{g}\}$ in $S$ forms a symplectic basis of $H_{1}(\partial W)$ such that each $b^{W}_{i}$vanishes in $H_{1}(W)$. 
In this situation, we can compute the matrices $A$, $B$ using curves: 
The $(i,j)$-entry of $A$ is $\langle b^{W}_{i},a^{V}_{j}\rangle_{S}$, and the $(i,j)$-entry of $B$ is $\langle b^{W}_{i}, b^{V}_{j}\rangle_{S}$.  
\end{rmk}

\subsubsection{Representatives for $A^{p}(q)$}
 Let $M$ be a lens space of type $(p,q)$ for $p>1$. 
$M$ has a surgery presentation as in Figure~\ref{a^p_q_surg} and a Heegaard splitting $V\cup W$ as in Figure~\ref{a^p_q}. 
In Figure~\ref{a^p_q}, $V$ is the inner handlebody and $W$ is the outer handlebody,
 and a box containing $\frac{-q}{p}$ represents curves in it, a result of resolving the intersection points of horizontal $|q|$ lines and vertical $|p|$ lines so that they twist in left-hand (or right-hand) side if $\frac{-q}{p}>0$ (or $<0$, respectively). 
We take the $(p,-q)$-curve as $b^{V}_{1}$ and the $(0,1)$-curve as $b^{W}_{1}$ on the splitting torus. 
Note that the orientation of $L(p,q)$ coming from the standard one of $S^3$ corresponds to that coming from $V$. 
By computation, we see that $M$ is a rational homology 3-sphere. 
Let $a^{V}_1$ be the $(r,s)$-curve for $r,s$ such that $-rq-sp=1$. 
Then the matrices in \ref{cal_tor} are $A=\left( -r \right)$ and $B=\left(-p\right)$ by computing the algebraic intersections $\langle a^{V}_1, b^{W}_1 \rangle$ and $\langle b^{V}_1, b^{W}_1 \rangle$ on $\partial W$. Note that $\partial V=-\partial W$. 
Thus by Fact~\ref{cfh}, $H_1(M)=\mathbb{Z}/p\mathbb{Z}$ and the torsion linking form is $\left( \begin{array}{c}
   -\frac{r}{p}    \\
  \end{array} \right)$. 
Let $x$ be a generator of $H_1(M)=\mathbb{Z}/p\mathbb{Z}$ corresponding to the $1\times 1$-matrix $\left( \begin{array}{c}
   -\frac{r}{p}    \\
  \end{array} \right)$. 
Note that $qx$ is also a generator of $H_1(M)=\mathbb{Z}/p\mathbb{Z}$ since $p$ and $q$ are coprime. 
Under this new generator, the matrix representation of the torsion linking form is
 $\left( \begin{array}{c}
   -\frac{rq^{2}}{p}    \\
  \end{array} \right)=
\left( \begin{array}{c}
   \frac{q}{p}    \\
  \end{array} \right)$ since $rq \equiv -1$ mod $p$, the same as $A^{p}(q)$. 
Henceforth, $M(A^{p}(q))$ denotes $L(p,q)$. 

\begin{figure}[htbp]
 \begin{center}
  \includegraphics[width=25mm]{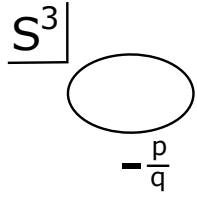}
 \end{center}
 \caption{A surgery description of $L(p,q)$}
 \label{a^p_q_surg}
\end{figure}

\begin{figure}[htbp]
 \begin{center}
  \includegraphics[width=125mm]{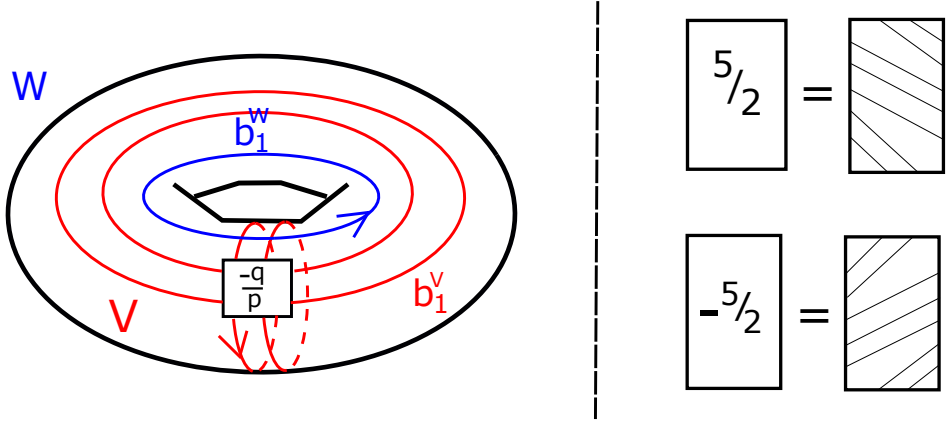}
 \end{center}
 \caption{left: A Heegaard splitting of $L(p,q)$\ \ \ \ right: Examples of boxes contain rational numbers}
 \label{a^p_q}
\end{figure}

\subsubsection{Representatives for $E^{k}_0$}
Let $M$ be a connected closed orientable 3-manifold having a surgery description as in Figure~\ref{e^k_0_surg}. 
It has a Heegaard splitting $V\cup W$ as in Figure~\ref{e^k_0}, so-called a vertical Heegaard splitting. 
The convention of boxes containing rational numbers is as in Figure~\ref{a^p_q}. 
Note that this Heegaard diagram is not minimally intersecting. 
We regard $V$ as the inner handlebody and $W$ as the outer handlebody. 
Note that the orientation coming from the standard one of $S^3$ corresponds to that coming from $V$. 
By computation, we see that $M$ is a rational homology 3-sphere. 
Let $a^{V}_1, a^{V}_2$ be curves as in Figure~\ref{e^k_0_a}, where $b^{V}_1, b^{V}_2$ is the same as in Figure~\ref{e^k_0}, which we abbreviate. 
Then the matrices in \ref{cal_tor} are 
   $A=\left( \begin{array}{cc}
   0       & 1  \\
   1  &    1   \\
  \end{array} \right)$
 and 
   $B=\left( \begin{array}{cc}
   2^{k}       & 0  \\
   -2^{k}  &    2^{k}   \\
  \end{array} \right)$. 
Thus by Fact \ref{cfh}, $H_1(M)=\mathbb{Z}\langle x, y \rangle/(2^{k}x-2^{k}y, 2^{k}y)$ and $\psi_{M}=-B^{-1}A=\left( \begin{array}{cc}
   0       & -2^{-k}  \\
   -2^{-k}  &    2^{1-k}   \\
  \end{array} \right)$ under the basis $\{x,y\}$. 
By changing a basis, $H_{1}(M)=\left( \mathbb{Z}\langle x+y\rangle/(2^{k}(x+y)) \right) \oplus \left( \mathbb{Z} \langle -x\rangle /(2^{k}(-x)\right)$
 and $\psi_{M}=\left( \begin{array}{cc}
   0       & 2^{-k}\\
   2^{-k}  &    0   \\
  \end{array} \right)$
 under the basis $\{x+y,-x\}$, the same as $E^{k}_0$. 
Henceforth, $M(E^{k}_0)$ denotes this $M$. 

\begin{figure}[htbp]
 \begin{center}
  \includegraphics[width=130mm]{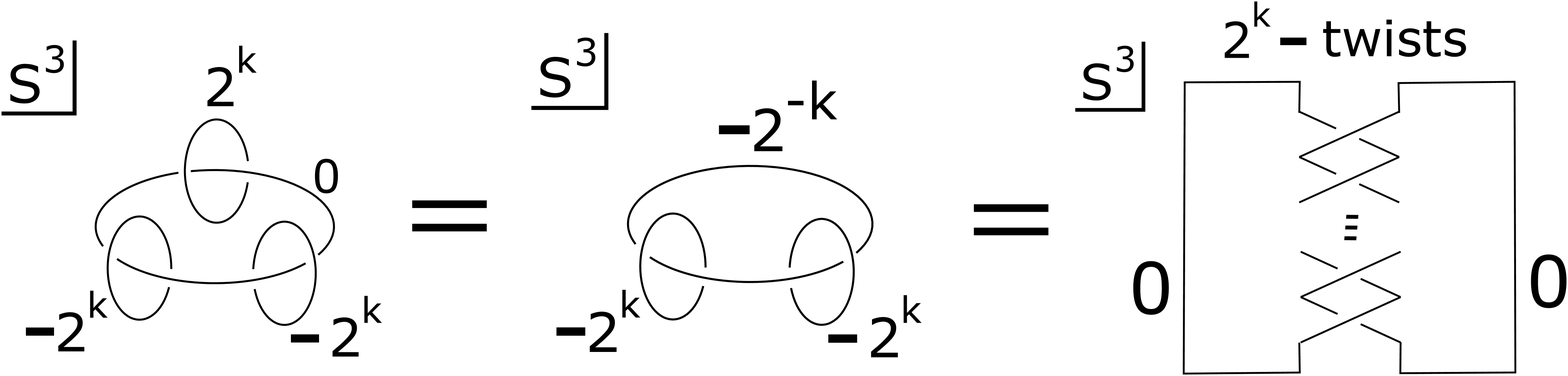}
 \end{center}
 \caption{Surgery descriptions of $M$}
 \label{e^k_0_surg}
\end{figure}

\begin{figure}[htbp]
 \begin{center}
  \includegraphics[width=80mm]{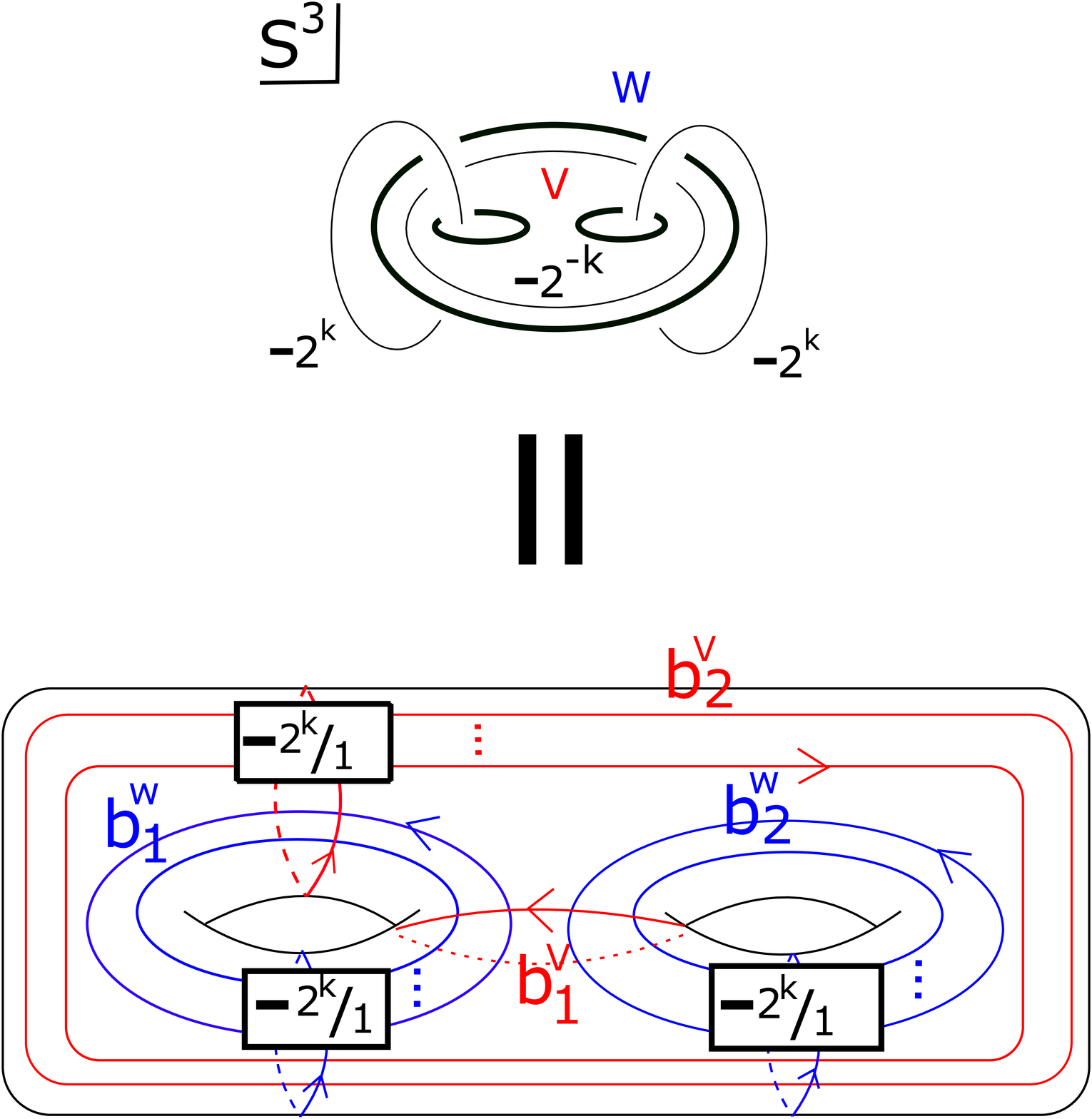}
 \end{center}
 \caption{A Heegaard splitting of $M$}
 \label{e^k_0}
\end{figure}

\begin{figure}[htbp]
 \begin{center}
  \includegraphics[width=80mm]{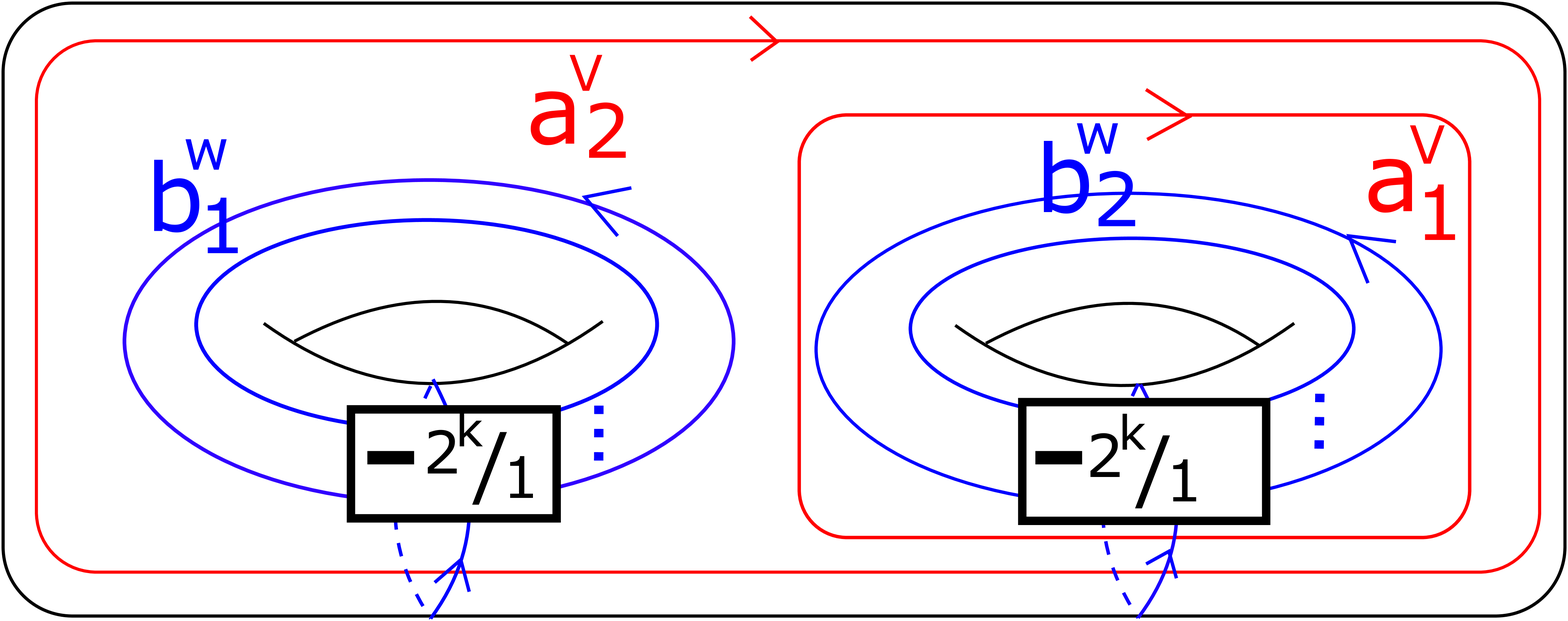}
 \end{center}
 \caption{Curves $a^{V}_1, a^{V}_2, b^{W}_{1}, b^{W}_{2}$}
 \label{e^k_0_a}
\end{figure}

\subsubsection{Representatives for $E^{k}_{1}$}
Let $M$ be a connected closed orientable 3-manifold having a surgery description as in Figure~\ref{e^k_1_surg}. 
It has a Heegaard splitting $V\cup W$ as in Figure~\ref{e^k_1}, so-called a vertical Heegaard splitting. 
The convention of boxes containing rational numbers is as in Figure~\ref{a^p_q}. 
We regard $V$ as the inner handlebody and $W$ as the outer handlebody. 
Note that the orientation coming from the standard one of $S^3$ corresponds to that coming from $V$. 
By computation, we see that $M$ is a rational homology 3-sphere. 
Let $a^{V}_1, a^{V}_2$ be curves as in Figure~\ref{e^k_1_a}, where $b^{V}_1, b^{V}_2$ are the same as in Figure~\ref{e^k_1}, which we abbreviate. 
Then the matrices in \ref{cal_tor} are 
   $A=\left( \begin{array}{cc}
   0       & -1  \\
   -3  &    3   \\
  \end{array} \right)$
 and 
   $B=\left( \begin{array}{cc}
   2^{k}       & 2\cdot2^{k}  \\
   -2^{k}  &   -3\cdot 2^{k}   \\
  \end{array} \right)$. 
Thus by Fact~\ref{cfh}, 
$H_1(M)=\mathbb{Z}\langle x, y\rangle/(2^{k}x-2^{k}y, 2\cdot2^{k}x-3\cdot2^{k}y)$ and $\psi_{M}=-B^{-1}A=\left( \begin{array}{cc}
   3\cdot 2^{1-k}       & -3\cdot 2^{-k}  \\
   -3\cdot 2^{-k}  &    2^{1-k}   \\
  \end{array} \right)$ under the basis $\{x,y\}$. 
By changing a basis, $H_{1}(M)=\left( \mathbb{Z}\langle x+y\rangle/(2^{k}(x+y)) \right) \oplus \left( \mathbb{Z} \langle -y\rangle /(2^{k}(-y)\right)$
 and $\psi_{M}=\left( \begin{array}{cc}
   2^{1-k}       & 2^{-k}\\
   2^{-k}  &    2^{1-k}   \\
  \end{array} \right)$
 under the basis $\{x+y,-y\}$, the same as $E^{k}_1$. 
Henceforth, $M(E^{k}_1)$ denotes this $M$. 

\begin{figure}[htbp]
 \begin{center}
  \includegraphics[width=130mm]{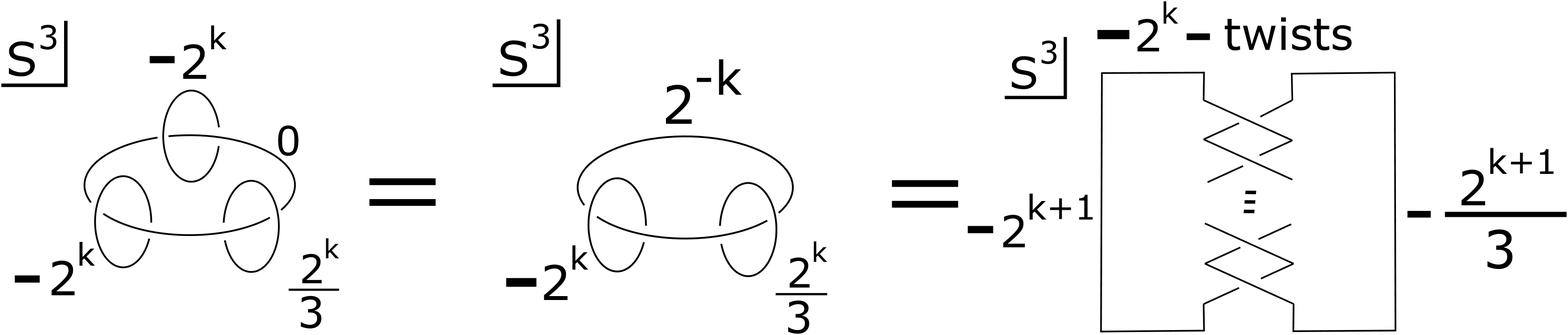}
 \end{center}
 \caption{Surgery descriptions of $M$}
 \label{e^k_1_surg}
\end{figure}

\begin{figure}[htbp]
 \begin{center}
  \includegraphics[width=80mm]{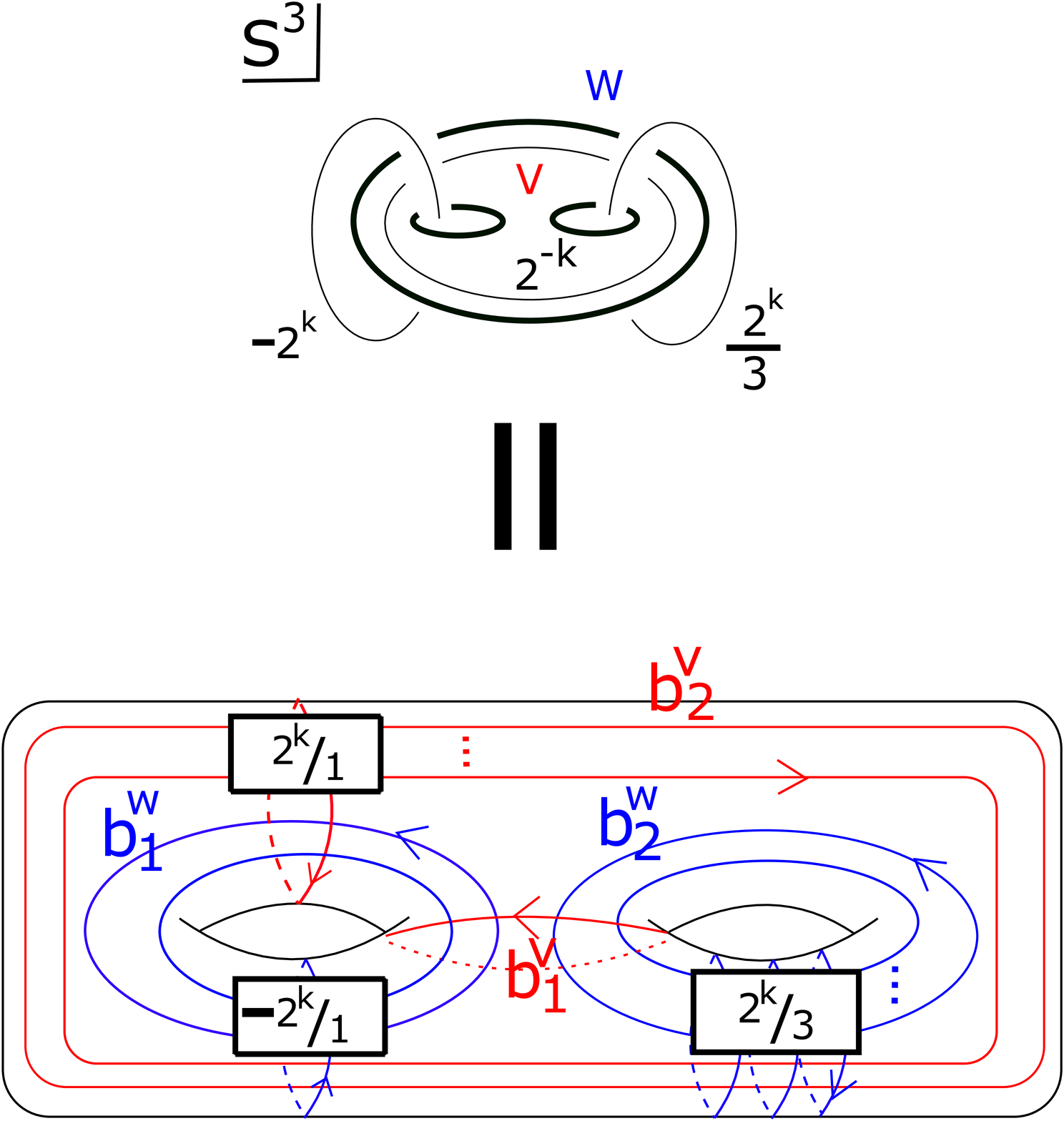}
 \end{center}
 \caption{A Heegaard splitting of $M$}
 \label{e^k_1}
\end{figure}

\begin{figure}[htbp]
 \begin{center}
  \includegraphics[width=100mm]{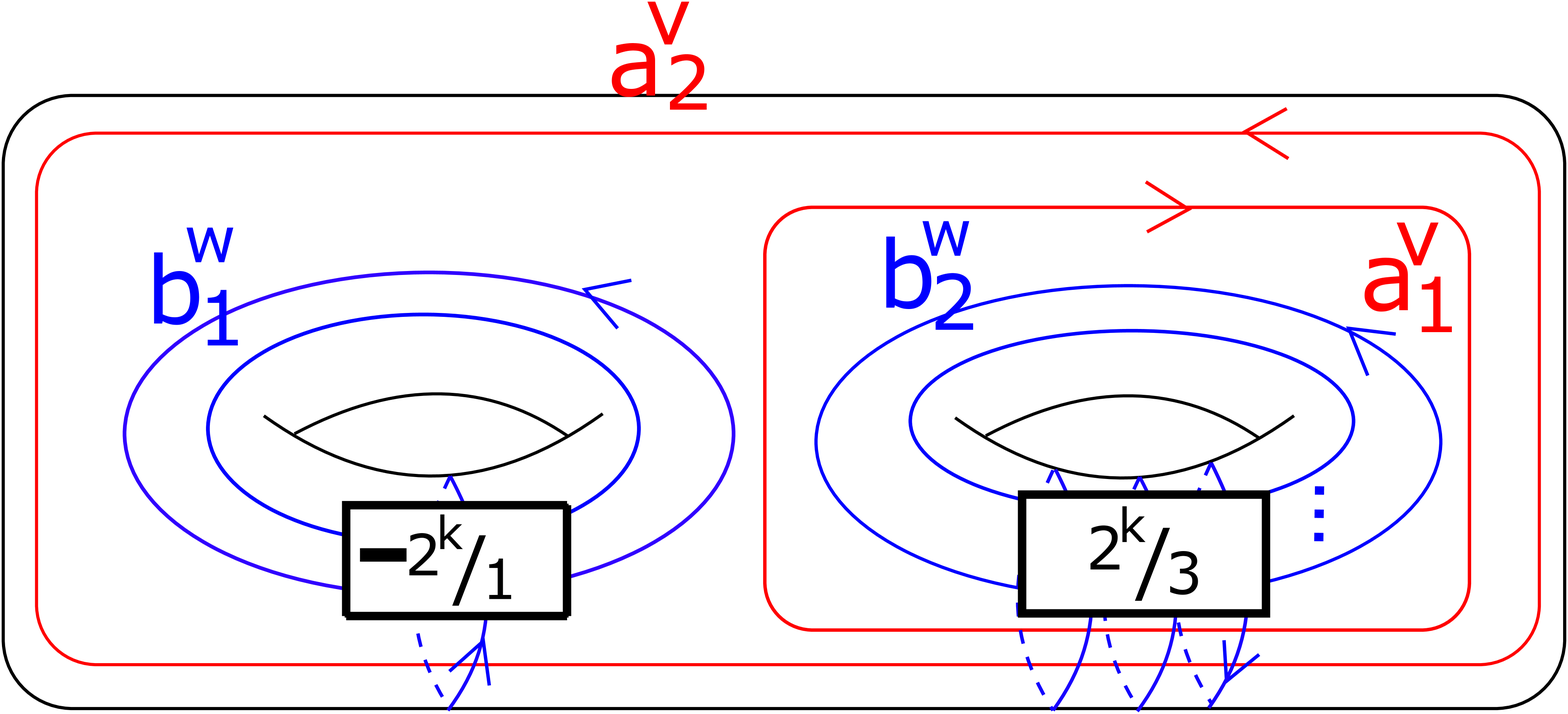}
 \end{center}
 \caption{Curves $a^{V}_1, a^{V}_2, b^{W}_{1}, b^{W}_{2}$}
 \label{e^k_1_a}
\end{figure}

\vspace{1.0cm}

From the above, we have representatives for connected closed oriented 3-manifolds with respect to the first homology groups and their torsion linking forms. 

\begin{prop} \label{representative}
Suppose that a connected closed oriented 3-manifold $M$ whose free part of the first homology group is isomorphic to $\mathbb{Z}^{r}$ and whose torsion linking form is equivalent to $(\bigoplus_{0\leq l\leq a}A^{p_l}(q_l))\oplus (\bigoplus_{0\leq i \leq e_0} E^{k_i}_{0}) \oplus (\bigoplus_{0\leq j\leq e_1}E^{k'_{j}}_1)$ for integers $a, p_{l}, q_{l}, e_{0}, k_{i}, e_{1}, k'_{j}$ as in the end of Subsection \ref{linkingpair}. \\
Then a 3-manifold $\left(\#^{r}S^2\times S^1\right)\#\left(\#_{0\leq l\leq a}M\left(A^{p_l}(q_l)\right)\right)\# (\#_{0\leq i \leq e_0} M(E^{k_i}_{0})) \# (\#_{0\leq j\leq e_1}M(E^{k'_{j}}_1))$ has isomorphic first homology  group to that of $M$ and has equivalent torsion linking forms to that of $M$, where as a notation we regard $\#^{0}S^2\times S^1$, $M(A^{1}(q_0))$, $M(E^{0}_{0})$ and $M(E^{0}_{1})$ as $S^3$.
\end{prop}

We use this in Section~\ref{proof}.

\section{Thickened surface and its spine bands} \label{surf_bands}
In this section, we explain how surfaces (with their spines) in a given 3-manifold correspond to pairs of annuli and bands in the manifold. 
All surfaces and 3-manifolds in this section are oriented, and we can use the terminology ``front'' and ``back'' sides of surfaces.\\\\
We state the statement as a proposition:
\begin{prop}\label{correspondence}
Let $M$ be an oriented 3-manifold. Fix non-negative integers $g$ and $n$.\\
There is a bijection between the set of isotopy classes of oriented surfaces in $M$, each of which is homeomorphic to $\Sigma_{g,n+1}$ with the spine as in Figure \ref{spine}
 and the set of oriented annuli $A_{1}, \dots,A_{2g+n}$ with oriented core curves and the oriented bands (regarded as small rectangles) $B_{1},\dots, B_{g}, C_{1}, \dots ,C_{g+n-1}$ satisfying the following condition (see Figure~\ref{bands} for example):
\begin{itemize}
\item One side of $\partial B_{i}$, which consists of four sides, is on the back side of $A_{2i-1}$ and is a properly embedded essential arc on $A_{2i-1}$, and the opposite side of $\partial B_{i}$ is on the front side of $A_{2i}$ and is a part of the core curve of $A_{2i}$ for $1\leq i\leq g$.

\item For $1\leq i\leq g$, near the intersection of $B_i$ with $A_{2i-1}$, the front side of $B_i$ is on the positive direction with respect to the orientation of the core curve of $A_{2i-1}$,
 and near that of with $A_{2i}$, the front side of $B_{i}$ is the ``right'' side of the core curve of $A_{2i}$, where we look the intersection (this is on the front side of $A_{2i}$) so that the core curve of $A_{2i}$ runs from the bottom to the top. 

\item For $1\leq i \leq g-1$, $C_{i}$ connects the ``left'' boundary of $A_{2i}$ and the ``right'' boundary of $A_{2i+1}$ so that the front sides are attached,
 and for $g\leq j\leq g+n-1$, $C_{j}$ connects the ``left'' boundary of $A_{j+g}$ and the ``left'' boundary of $A_{j+g+1}$ so that the front sides are attached, where we look the front side of $A_{k}$ so that the core curve runs from the bottom to the top.

\end{itemize}
\end{prop}

We call the condition for annuli and bands in Proposition~\ref{correspondence} the condition $\spadesuit$. 
We construct oriented annuli and bands from a surface with its spine in Subsection~\ref{fromsurf}, and construct a surface with its spine from oriented annuli and bands in Subsection~\ref{fromknots}. 
We can see that these operations are the inverses of each other. 
Thus Proposition~\ref{correspondence} holds.

\begin{figure}[htbp]
 \begin{center}
  \includegraphics[width=150mm]{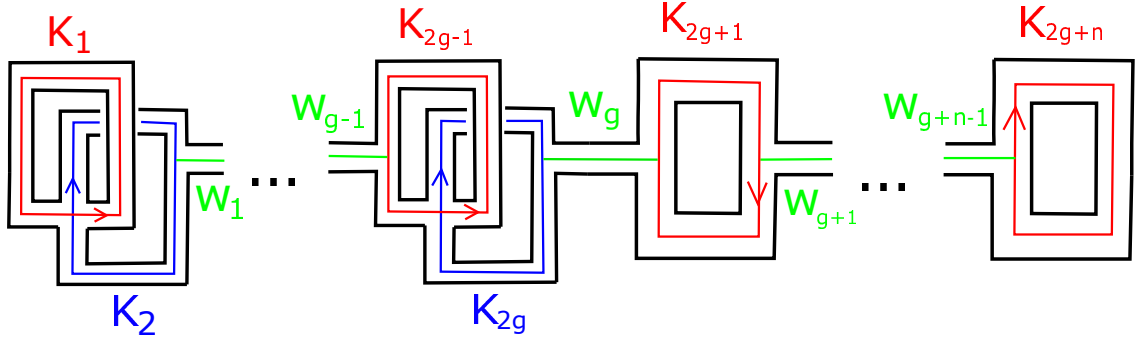}
 \end{center}
 \caption{A spine of $\Sigma_{g,n+1}$}
 \label{spine}
\end{figure}

\begin{figure}[htbp]
 \begin{center}
  \includegraphics[width=150mm]{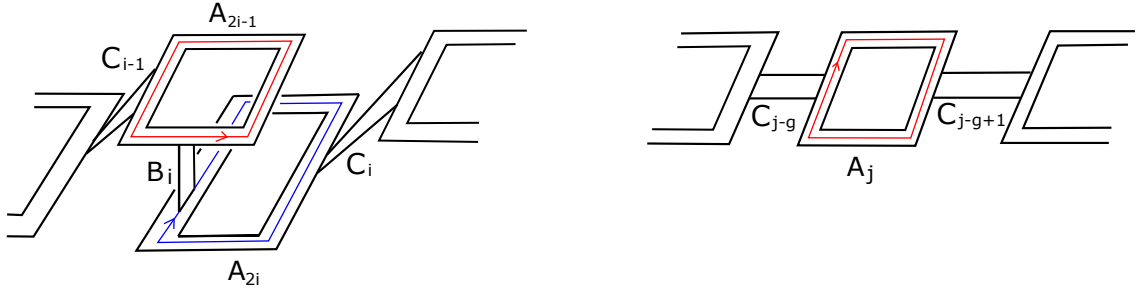}
 \end{center}
 \caption{Annuli and bands}
 \label{bands}
\end{figure}

\subsection{Annuli and bands obtained from surfaces}\label{fromsurf}
Let $F$ be a surface homeomrphic to $\Sigma_{g,n+1}$ in $M$. 
Fix an embedding $\iota : \Sigma_{g,n+1}\times[-1,1] \longrightarrow M$ such that $\iota(\Sigma_{g,n+1}\times \{0\})=F$ and identify $\Sigma_{g,n+1}$ with $F$ via $\iota$. 
Take a spine $(K_{1}\cup K_{2})\cup w_{1}\cup \cdots \cup w_{g-1}\cup (K_{2g-1}\cup K_{2g})\cup w_{g}\cup K_{2g+1} \cup w_{g+1} \cup \cdots \cup w_{g+n-1} \cup K_{2g+n}$ of $F$ as in Figure~\ref{spine}. 
We assign an orientation to $K_i$ for each $1\leq i\leq 2g+n$. 
Denote the point $K_{2i-1}\cap K_{2i}$ for $1\leq i\leq g$ by $p_i$. 
Let $U_i$ be the square centered at $p_i$ and $V_j$ the rectangle around $w_j$ as in Figure~\ref{charts}. 
Take charts $\phi_{i} : [-1,1]^{2} \longrightarrow U_i$ such that $\phi([-1,1]\times \{0\})$ is a part of $K_{2i-1}$, $\phi(\{0\}\times [-1,1])$ is a part of $K_{2i}$, and $\psi_{j} : [-1,1]^{2} \longrightarrow V_{j}$ such that $\psi([-1,1]\times \{0\})$ is a part of $w_j$ and $\{-1\}\times [-1,1]$ is near $K_{2j}$ for $1\leq j\leq g$, near $K_{j+g}$ for $g+1\leq j\leq g+n-1$.

\begin{figure}[htbp]
 \begin{center}
  \includegraphics[width=70mm]{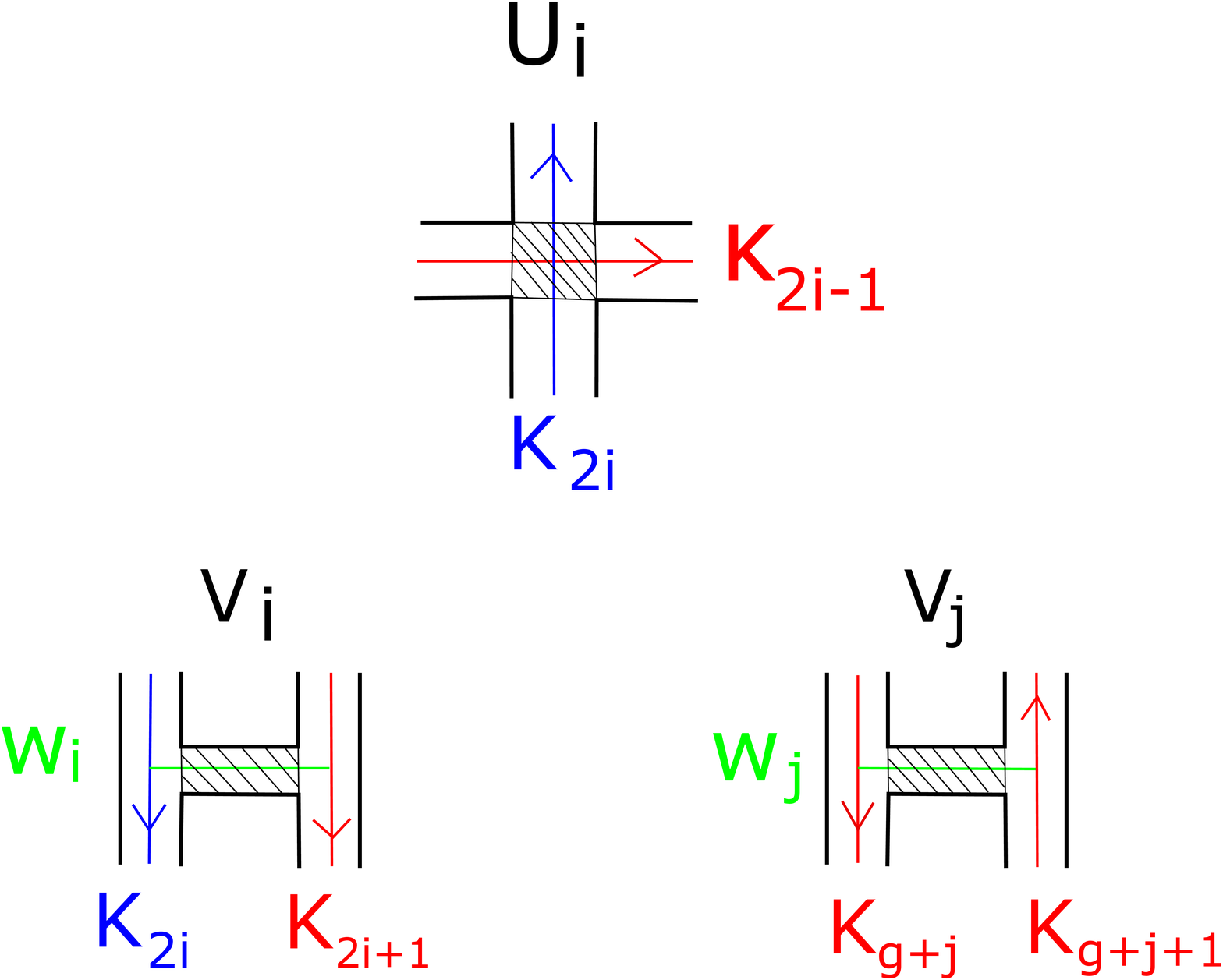}
 \end{center}
 \caption{Charts}
 \label{charts}
\end{figure}

Denote the annulus whose core curve is $K_i$ by $A'_i$ for $1\leq i\leq 2g+n$ as in Figure~\ref{annuli}. 
Consider $F\times [-1,1]$, see Figure \ref{thicken1}. We regard that it is in $M$ via $\iota$. 
Let $A_{2i-1}$, $A_{2i}$ and $A_{j}$ be $A'_{2i-1}\times \{1\}$, $A'_{2i}\times \{-1\}$ and $A'_{j}\times \{1\}$, respectively for $1\leq i\leq g$ and $2g+1\leq j\leq 2g+n$. 
These annuli can be regarded as framed knots in $M$. 
Let $B_{i}$, $C_{i}$ and $C_{j}$ be $\phi_{i} \left( \{0\}\times[-1,1] \right) \times[-1,1]$, $\left\{\psi_{i}(\{t\}\times [-1,1])\times \{t\}| t\in [-1,1] \right\}$ and $\psi_{j}([-1,1]^{2})\times \{1\}$, respectively for $1\leq i\leq g$ and $g+1\leq j\leq g+n-1$. 
These are bands in $M$. 
We get an oriented annuli $A_{k}$'s and bands $B_{i}$'s and $C_{j}$'s (these are rectangles) in $M$ as in Figure~\ref{bands}. 
Observe that these satisfy the condition $\spadesuit$ with appropriate orientations.

\begin{figure}[htbp]
 \begin{center}
  \includegraphics[width=80mm]{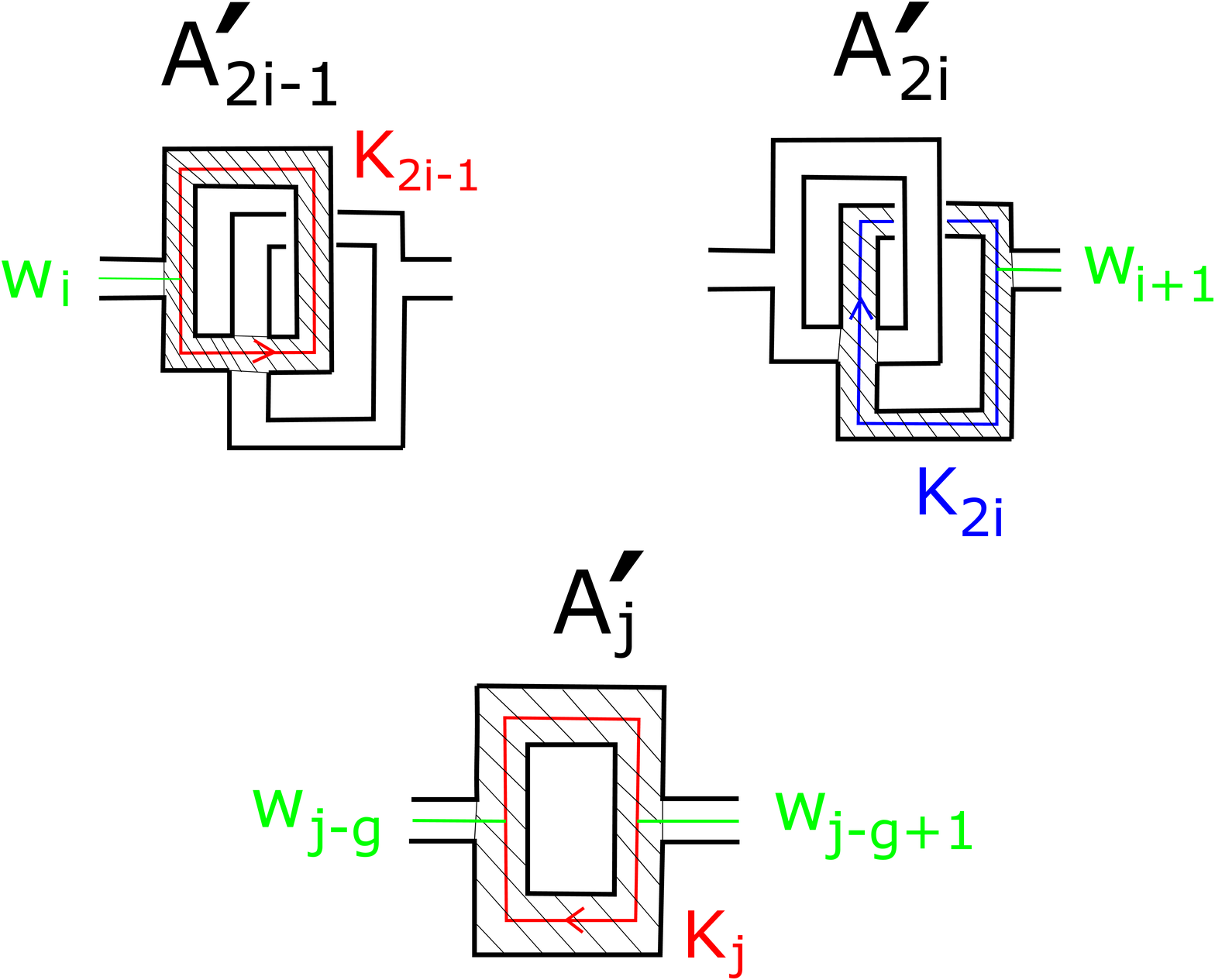}
 \end{center}
 \caption{Annuli}
 \label{annuli}
\end{figure}

\begin{figure}[htbp]
 \begin{center}
  \includegraphics[width=150mm]{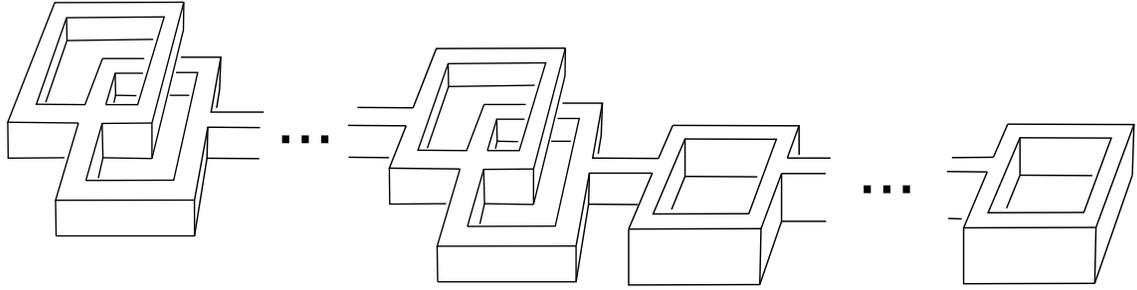}
 \end{center}
 \caption{Thickened $F$}
 \label{thicken1}
\end{figure}

\subsection{Surfaces from annuli and bands} \label{fromknots}
Let $A_{1}, \dots, A_{2g+n}$ be oriented annuli and $B_{1},\dots, B_{g}, C_{1}, \dots ,C_{g+n-1}$ be bands satisfying the condition $\spadesuit$ in $M$. 
These look like some embeddings of the union of annuli and bands as in Figure \ref{bands}. 
%
%
Thicken these (denoted by $N$) and consider curves with colors as in Figure \ref{thicken2}. In this figure, a part of the core curve of $A_{2i}$, which is also one edge of the band $B_{i}$ is replaced with the other three edges of $B_{i}$ for each $1\leq i\leq g$. 
A small neighborhood $F$ of these curves in $\partial N$ is homeomorphic to $\Sigma_{g,n+1}$ with a spine as in Figure~\ref{spine}. 
Let $K_{2i-1}$ denote the red curve coming from $A_{2i-1}$ and $K_{2i}$ the blue curve coming from $A_{2i}$ and $B_{i}$ for $1\leq i\leq g$, and $w_{j}$ the green curve coming from $C_{j}$ for $1\leq j\leq g+n-1$. 
For each $1\leq i \leq 2g+n$, fix an essential arc $\alpha _{i}$ in $A_{i}$ such that it is disjoint from $w_{j}$ for $1\leq j\leq g+n-1$ and $K_{l}$ for $l\neq i$, and let $D_{i}$ be an essential disk in $N$ obtained by thickening $\alpha_{i}$. 
Then each $\partial D_{i}$ intersects $\partial F$ twice with the opposite orientations and $D_{i}$'s cut $N$ into a ball. 
Therefore, $F$ gives $N$ the structure of $\Sigma_{g,n+1}\times[-1,1]$ so that $F$ is $\Sigma_{g,n+1}\times \{1\}$. 
The colored curves play a role of a spine of $\Sigma_{g,n+1}$. 
We get a surface corresponding to $\Sigma_{g,n+1}\times \{0\}$ with a spine obtained by projecting the curves along the product structure. 
In other word, the surface $\Sigma_{g,n+1}\times \{0\}$ with its spine is isotopic to $F$ with its spine in $M$. 
We can easily show that the operations in Subsections~\ref{fromsurf} and \ref{fromknots} are inverses of each other.

\begin{rmk}
Let $m_{A_{2i-1}}$ be a meridian curve of $A_{2i-1}$ (i.e. an oriented circle on the boundary of a small neighborhood of $A_{2i-1}$ which bounds a disk in the neighborhood intersecting the core of $A_{2i-1}$ once positively) for $1\leq i\leq g$. 
We assume $m_{A_{2i-1}}$ is in a small neighborhood of $A_{2i-1}$ and disjoint from $A_{2i-1}$. 
Note that the curve obtained from $A_{2i}$ by the above procedure represents $[A_{2i}]-[m_{A_{2i-1}}]$
 in $H_{1}(M\setminus N)$, where $[A_{2i}]$ and $[A_{2i-1}]$ denote the elements represented by push-ups of the core curves of $A_{2i}$ and $A_{2i-1}$, respectively. 
\end{rmk}

\begin{figure}[htbp]
 \begin{center}
  \includegraphics[width=150mm]{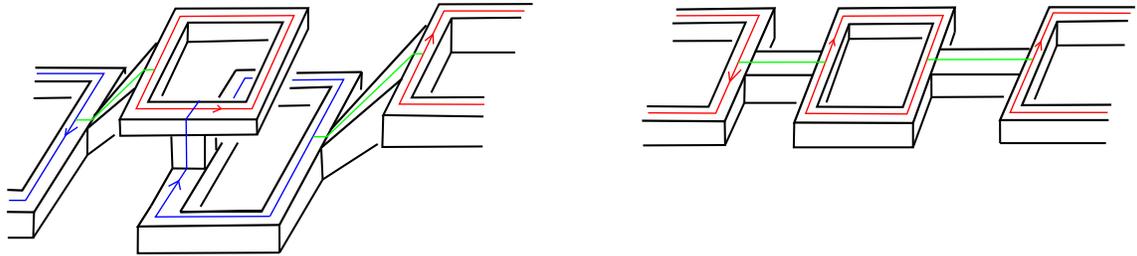}
 \end{center}
 \caption{Thickened annuli and bands}
 \label{thicken2}
\end{figure}

\section{A proof of Theorem~\ref{thm}} \label{proof}
In this section, we give a proof of Theorem~\ref{thm}. 
At first, we consider the condition for a connected closed orientable 3-manifold having a homologically fibered link whose homological fiber is given homeomorphic type using a surgery diagram. 
In order to define the linking number $lk(\cdot,\cdot)$ among knots in the surgery link, we have to give the link an orientation. 
We fix any one. However, this choice does not matter for the existence of a homologically fibered link whose fiber is the given homeomorphism type, see Remark \ref{choice}.

\begin{prop} \label{condition_surgery}
Let $M$ be a connected closed orientable 3-manifold. 
Suppose that $M$ is obtained from $S^3$ by the surgery along a link $L=L_{1}\sqcup \cdots \sqcup L_{m}$, where each $L_{i}$ is an oriented knot and its coefficient is $\frac{p_i}{q_i}$.\\
Then $M$ has a homologically fibered link whose homological fiber is homeomorphic to $\Sigma_{g,n+1}$ if and only if there exist $m \times (2g+n)$-matrix of integer coefficients $X$ and $(2g+n) \times (2g+n)$-symmetric matrix of integer coefficients $Y$ satisfying the following equation:
\begin{eqnarray} \label{surg}
   \left| \begin{array}{cc}
     \Phi    &  \Psi X  \\
      X^{t}    &  Y+(\mathcal{E}\oplus O_{n}) \\
  \end{array} \right| = \pm 1,
\end{eqnarray}
where $\Phi$ is an $m\times m$-matrix whose $(i,i)$-entry is $p_i$ and $(i,j)$-entry is $q_{i}lk(L_{i},L_{j})$ for $i\neq j$, $\Psi$ is an $m\times m$-diagonal matrix whose $(i,i)$-entry is $q_{i}$, and $\mathcal{E}$ is a $2g\times 2g$-matrix whose $(2i-1,2i)$-entry is $1$ for every $1\leq i\leq g$ and the others are $0$.
\end{prop}

\begin{proof}
Suppose $M$ has a homologically fibered link whose homological fiber $F$ is homeomorphic to $\Sigma_{g,n+1}$. We will construct a solution of (\ref{surg}). 
By choosing a spine of $F$, and using the construction in Section \ref{surf_bands}, 
we get oriented annuli $A_{1},\dots,A_{2g+n}$ and oriented bands $B_{1},\dots,B_{g}$, $C_{1},\dots, C_{n}$ satisfying the condition $\spadesuit$. We regard these annuli and bands are in the surgery diagram on $S^3$ missing the surgery link. 
Let $T_{i}$ be the canonical longitude (regarded in $S^3$) of the core curve of $A_{i}$ and $m_{A_{i}}$ the meridian of (the core curve of) $A_{i}$ such that $lk(T_{i},m_{A_{i}})=1$ and $t_{i}$ the framing number of $A_{i}$ regarded as a framed knot. 
Note that $H_{1}(M\setminus F)\cong H_{1}\left( M\setminus \left( (\bigcup_{1\leq i \leq 2g+n}A_{i})\cup(\bigcup_{1\leq j \leq g}B_{j}))\cup(\bigcup_{1\leq l \leq n}C_{l}) \right) \right) \cong H_{1}\left( M\setminus(\bigcup_{1\leq i \leq 2g+n}A_{i}) \right)$, and that the push-up of $K_{i}$, which is in spines of $F$ as in Figure \ref{spine} represents\\
$\begin{cases}
[T_{2i'-1}]+t_{2i'-1}[m_{A_{2i'-1}}]\ \ {\rm if} \ 1\leq i \leq 2g \ {\rm and}\ i=2i'-1,\\
[T_{2i'}]+t_{2i'}[m_{A_{2i'}}]-[m_{A_{2i'-1}}]\ \ {\rm if} \ 1\leq i \leq 2g \ {\rm and}\ i=2i',\\
[T_{i}]+t_{i}[m_{A_{i}}]\ \ {\rm if} \ 2g+1\leq i \leq 2g+n\\
\end{cases}$\\
in $H_{1}\left( M\setminus(\bigcup_{1\leq i \leq 2g+n}A_{i})\right)$. \\
Let $m_{L_{i}}$ be the meridian of $L_{i}$ and set $l_{i,j}$ to be $lk(L_{i},L_{j})$ and $\tilde{l}_{i,j}$ to be $lk(L_{i},T_{j})$ and $\bar{l}_{i,j}$ to be $lk(T_{i},T_{j})$. 
Note that $H_{1}\left( S^{3}\setminus \left(L\cup (\bigcup_{1\leq i \leq 2g+n}A_{i})\right)\right)\cong \mathbb{Z}^{m+2g+n}$ is generated by $\{ [m_{L_{1}}],\dots,[m_{L_m}], [m_{A_1}],\dots,[m_{A_{2g+n}}]\}$, and that $\frac{p_i}{q_i}$-slope of $L_i$ represents an element $p_{i}[m_{L_{i}}]+q_{i}\{ \Sigma_{\substack{1\leq j\leq m\\ i\neq j}} l_{i,j}[m_{L_{j}}]+\Sigma_{1\leq k \leq 2g+n} \tilde{l}_{i,k}[m_{A_{k}}]\}$ in $H_{1}\left( S^{3}\setminus \left(L\cup (\bigcup_{1\leq i \leq 2g+n}A_{i})\right)\right)$, denoted by $\mathbb{A}_{i}$. 
Note also that the push-up of $K_{i}$, which is in a spine of $F$ as in Figure \ref{spine} represents \\
$\begin{cases}
t_{2i'-1}[m_{A_{2i'-1}}]+\Sigma_{1\leq j\leq m}\tilde{l}_{j,i}[m_{L_j}]+\Sigma_{\substack{1\leq k \leq 2g+n\\ k\neq i}} \bar{l}_{i,k}[m_{A_{k}}]\ \ {\rm if} \ 1\leq i \leq 2g \ {\rm and}\ i=2i'-1,\\
t_{2i'}[m_{A_{2i'}}]-[m_{A_{2i'-1}}]+\Sigma_{1\leq j\leq m}\tilde{l}_{j,i}[m_{L_j}]+\Sigma_{\substack{1\leq k \leq 2g+n\\ k\neq i}} \bar{l}_{i,k}[m_{A_{k}}]\ \ {\rm if} \ 1\leq i \leq 2g \ {\rm and}\ i=2i',\\
t_{i}[m_{A_i}]+\Sigma_{1\leq j\leq m}\tilde{l}_{j,i}[m_{L_j}]+\Sigma_{\substack{1\leq k \leq 2g+n\\ k\neq i}} \bar{l}_{i,k}[m_{A_{k}}]\ \ {\rm if} \ 2g+1\leq i \leq 2g+n\\
\end{cases}$\\
under the basis $\{ [m_{L_{1}}],\dots,[m_{L_m}], [m_{A_1}],\dots,[m_{A_{2g+n}}]\}$ of $H_{1}\left(S^{3}\setminus \left(L\cup (\bigcup_{1\leq i \leq 2g+n}A_{i})\right)\right)$, denoted by $\mathbb{B}_{i}$. 
Since $F$ is homological fiber, $H_{1}(M\setminus F)$ must be $\mathbb{Z}^{2g+n}$ and the push-ups of $K_{i}$'s form a free basis of $H_{1}(M\setminus F)$. 
This implies that $\{ [m_{L_{1}}],\dots,[m_{L_m}], [m_{A_1}],\dots,[m_{A_{2g+n}}]\}/\langle \mathbb{A}_{1},\dots ,\mathbb{A}_{m} \rangle$ is isomorphic to $\mathbb{Z}^{2g+n}$, which also implies that $\mathbb{A}_{1},\dots ,\mathbb{A}_{m}$ are linearly independent in $H_{1}\left(S^{3}\setminus \left(L\cup (\bigcup_{1\leq i \leq 2g+n}A_{i})\right) \right)$, and $\mathbb{B}_{j}$'s modulo $\langle \mathbb{A}_{1},\dots ,\mathbb{A}_{m} \rangle$ form a basis of it. 
Thus $\{ \mathbb{A}_{1},\dots,\mathbb{A}_m, \mathbb{B}_{1},\dots,\mathbb{B}_{2g+n}\}$ is a basis of $\mathbb{Z}^{m+2g+n}$ freely generated by $\{ [m_{L_{1}}],\dots,[m_{L_m}], [m_{A_1}],\dots,[m_{A_{2g+n}}]\}$. 
Thus considering the change of basis matrix, we get a solution of (\ref{surg}) by setting $X$ equal to a matrix whose $(i,j)$-entry is $\tilde{l}_{i,j}$, $Y$ equal to a symmetric matrix whose $(k,k)$-entry is $t_{k}$, $(2i-1,2i)$-entry and $(2i,2i-1)$-entry for $i\leq g$ are $\bar{l}_{2i-1,2i}-1$ and other $(i,j)$-entry is $\bar{l}_{i,j}$. \\

Conversely, suppose we have a solution $X$, $Y$ for (\ref{surg}). 
Then we can construct oriented framed knots $K_{1},\dots,K_{2g+n}$ in the surgery diagram on $S^3$, which will be the core curves of annuli $A_{1},\dots , A_{2g+n}$, such that $lk(L_{i},K_{j})$ is the $(i,j)$-entry of $X$, the framing of $K_{i}$ is the $(i,i)$-entry of $Y$, $lk(K_{2i-1},K_{2i})$ for $i\leq g$ is equal to $1$ plus ${\rm the}\ (2i-1,2i)$-entry of $Y$, and $lk(K_{i},K_{j})$ for other pair $(i,j)$ is the $(i,j)$-entry of $Y$. 
Take bands $B_{1},\dots, B_{g}, C_{1},\dots ,C_{2g+n-1}$ satisfying the condition $\spadesuit$. 
Note that we can choose bands arbitrarily since these do not affect on the homology. 
We get a surface which is homeomorphic to $\Sigma_{g,n+1}$ by the construction in Section~\ref{surf_bands}. 
By reversing the argument in the previous paragraph, we see that this surface is a homological fiber.
\end{proof}

\begin{rmk}\label{choice}
Suppose there is a solution $X$, $Y$ for the equation (\ref{surg}) in Proposition~\ref{condition_surgery} for some orientation of the surgery link. 
If we reverse the orientation of the $i$-th component of $m$-component surgery link, then the signs of the $i$-th row and $i$-th column of $\Phi$ except for the $(i,i)$-entry and the $(i,i)$-th entry of $\Psi$ in the left hand side of the equation (\ref{surg}) are switched. 
Then a pair of the matrix obtained from $X$ by switching the signs of the $i$-th row and $Y$ is a solution of the new equation. 
Therefore the existence of a solution of the equation (\ref{surg}) is preserved. 
\end{rmk}

\begin{rmk}
For a given surgery diagram and a solution of (\ref{surg}), we can concretely construct a homologically fibered link as in the latter part of the proof of Proposition~\ref{condition_surgery}. 
However we have many choices in construction: There are many links with linking each others in the given number and linking the surgery link in the given number, and we can arbitrary take bands. 
These homologically fibered links corresponding to one solution are not always different, i.e. they may be isotopic. 
Moreover, two homological fibered links corresponding to two different solution may be isotopic. 
\end{rmk}

\begin{proof}
[Proof of Theorem~\ref{thm}]
Suppose that $M$ is a connected closed oriented 3-manifold whose free part of the first homology group is isomorphic to $\mathbb{Z}^{r}$ and whose torsion linking form is equivalent to $(\bigoplus_{0\leq l\leq a}A^{p_l}(q_l))\oplus (\bigoplus_{0\leq i \leq e_0} E^{k_i}_{0}) \oplus (\bigoplus_{0\leq j\leq e_1}E^{k'_{j}}_1)$. \\
Then $\tilde{M}= \left(\#^{r}S^2\times S^1\right)\#\left(\#_{0\leq l\leq a}M\left(A^{p_l}(q_l)\right)\right)\# (\#_{0\leq i \leq e_0} M(E^{k_i}_{0})) \# (\#_{0\leq j\leq e_1}M(E^{k'_{j}}_1))$ has the isomorphic first homology  group to $M$ and has a torsion linking form equivalent to that of $M$ by Proposition~\ref{representative}. 
By Proposition~\ref{dependence}, $M$ has a homologically fibered link whose homological fiber is homeomorphic to $\Sigma_{g,n+1}$ if and only if $\tilde{M}$ has such homologically fibered link. 
We have a surgery diagram on $S^3$ for $\tilde{M}$ which consists of $r$ copies of an unknot with surgery slope $0$ (for $\#^{r}S^2\times S^1$), $a$ copies of an unknot with surgery slope $\frac{p_l}{-q_l}$ for $0\leq l\leq a$ (for $\#_{0\leq l\leq a}M\left(A^{p_l}(q_l)\right)$), $e_0$ copies of $2$-component link as in the right of Figure \ref{e^k_0_surg} for $0\leq i \leq e_0$ (for $\#_{0\leq i \leq e_0} M(E^{k_i}_{0})$), and $e_1$ copies of $2$-component link as in the right of Figure \ref{e^k_1_surg} for $0\leq j \leq e_1$ (for $\#_{0\leq j\leq e_1}M(E^{k'_{j}}_1)$). 
Applying Proposition~\ref{condition_surgery} with this surgery diagram with appropriate order, we get the equality (\ref{condition}).
\end{proof}

Almost the same statement as the next corollary was given by Nozaki.

\begin{cor}
Let $M$ be a connected closed orientable 3-manifold. 
Suppose that $M\#(S^{2}\times S^1)$ has a homologically fibered link whose homological fiber is homeomorphic to $\Sigma_{0,n+1}$ for $n\geq 1$ (resp. $\Sigma_{1,1}$). 
Then $M$ has a homologically fibered link whose homological fiber is homeomorphic to $\Sigma_{0,n}$ (resp. $\Sigma_{0,2}$).
\end{cor}

\begin{proof}
At first, for the case where $M\#(S^2\times S^1)$ has a homologically fibered link whose homological fiber is homeomorphic to $\Sigma_{0,2}$, we can see that $H_{1}(M\#(S^2\times S^1))\cong H_{1}(M)\oplus H_{1}(S^2\times S^1)$ is generated by one element by Remark \ref{lowerbound}. 
Thus $M$ is an integral homology $3$-sphere and has a homologically fibered link whose homological fiber is homeomorphic to $\Sigma_{0,1}$. 
In the following, for the case where $M\#(S^2\times S^1)$ has a homologically fibered link whose homological fiber is homeomorphic to $\Sigma_{0,n+1}$, we assume that $n\geq 2$.\\
Fix a surgery diagram $\mathcal{D}$ of $M$, and let $m-1$ be the number of knots in $\mathcal{D}$. 
A surgery diagram $\mathcal{D}'$ of $M\#(S^{2}\times S^1)$ is obtained by adding disjoint unknot $U$ with surgery slope $0(=\frac{0}{1})$ to $\mathcal{D}$. 
We give an order among the link in $\mathcal{D}'$ so that $U$ is the first component. 
Then by Proposition \ref{condition_surgery}, we have a solution of (\ref{surg}) for $\Sigma_{0,n+1}$ (resp. $\Sigma_{1,1}$) and fix it. 
Let $\mathbb{X}$ denote the matrix in the left hand side of (\ref{surg}) and let $\mathbb{X}_{i,j}$ denote the $(i,j)$-entry of $\mathbb{X}$. 
Note that $\mathbb{X}_{1,i}=\mathbb{X}_{i,1}=0$ for $1\leq i\leq m$, and that $\mathbb{X}_{1,i}=\mathbb{X}_{i,1}$ for $m+1\leq i\leq m+n$ (resp. $m+1 \leq i\leq m+2$). \\
Since (\ref{surg}) holds, $\mathbb{X}_{1,m+1}, \mathbb{X}_{1,m+2},\dots,\mathbb{X}_{1,m+n}$ (resp. $\mathbb{X}_{1,m+1}, \mathbb{X}_{1,m+2}$) are coprime. 
This implies that the (ordered) set $\{ \mathbb{X}_{1,m+1}, \mathbb{X}_{1,m+2},\dots,\mathbb{X}_{1,m+n}\}$ (resp. $\{ \mathbb{X}_{1,m+1}, \mathbb{X}_{1,m+2}\}$) can be changed into $n$-element set (resp. two-element set) such that the first element is $1$ or $-1$ and the other elements are $0$ in finitely many steps, at each of which the $l$-times of the $i$-th element is added to the $j$-th element for some integer $l$ and $1\leq i\neq j\leq n$ (resp. $1\leq i\neq j\leq 2$). 
Fix one of such sequence of steps. 
Along this sequence, at each step, say the $l$-times of the $i$-th element is added to the $j$-th element, we change $\mathbb{X}$ as follows: 
Add $l$ times the $(m+i)$-th column to the $(m+j)$-th column and then add $l$ times the $(m+i)$-th row to the $(m+j)$-th row. 
Note that at any time in the sequence, each matrix obtained from $\mathbb{X}$ is also a ``solution'' of (\ref{surg}). 
A solution for a homologically fibered link whose homological fiber is homeomorphic to neither $\Sigma_{0,n+1}$ nor $\Sigma_{1,1}$ is not necessarily preserved under this operation. 
At the end of the sequence we get a matrix $\mathbb{X}'$ obtained from $\mathbb{X}$ whose $(1,i)$-entry is $0$ for $i\neq m+1$ and $\pm1$ for $i=m+1$. 
Note that the $(i,1)$-entry of $\mathbb{X}'$ is $0$ for $i\neq m+1$ and $\pm1$ for $i=m+1$. \\
Expand the determinant of $\mathbb{X}'$ along the first row (the only one cofacter survives) and expand the determinant of the cofacter along the first column (the only one cofacter survives). 
The cofacter finally obtained is a solution for $\Sigma_{0,n}$ (resp. $\Sigma_{0,2}$) in a 3-manifold which has a surgery diagram $\mathcal{D}$, representing $M$.

\end{proof}

\section{Examples}\label{sec_examples}
Though we have Theorem~\ref{thm}, it is difficult in general to find a solution of (\ref{condition}) for a given manifold and homeomorphic type of a surface. 
Nozaki \cite{nozaki} proved that $L(p,q)$ has a homologically fibered link whose homological fiber is homeomorphic to $\Sigma_{1,1}$ for every pair $(p,q)$ by solving some equations, which is equivalent to (\ref{condition}), using the density theorem. 
The author \cite{sekino} proved that $L(p,q)$ has a homologically fibered link whose homological fiber is homeomorphic to $\Sigma_{0,3}$ for every pair $(p,q)$ by solving (\ref{condition}) following the Nozaki's argument. 
Thanks to Nozaki \cite{nozaki}, we see that ${\rm hc}(M(A^{p}(q)))=1$ for every $(p,q)$. 
Moreover, he also proved in his thesis that ${\rm hc} \left ((\#^{2m}S^2\times S^1)\#M(A^{p}(q))\right) =m+1$, and ${\rm hc}\left( (\#^{2m+1}S^2\times S^1)\#M(A^{p}(q))\right) =m+1$ if $q$ or $-q$ is quadratic residue modulo $p$ and ${\rm hc}\left( (\#^{2m+1}S^2\times S^1)\#M(A^{p}(q))\right) =m+2$ otherwise for every non-negative integer $m$. 
Thus we know ${\rm hc}(\cdot)$ for all 3-manifolds whose torsion linking form is $A^{p}(q)$. 
In this section, we determine ${\rm hc}(\cdot)$ for 3-manifolds whose torsion linking forms are the other generators for linkings, $(\#^{r}S^2\times S^1)\#M(E^{k}_{0})$ and $(\#^{r}S^2\times S^1)\#M(E^{k}_{1})$ for a non-negative integer $r$, we state as a proposition, proved at Subsections~\ref{hcek0} and~\ref{hcek1}:

\begin{prop}In the following, $\lceil n \rceil$ and $\lfloor n \rfloor$ for a real number $n$ represent the minimal integer greater than or equal to $n$ and the maximal integer less than or equal to $n$, respectively.
\begin{itemize}
\item For $r\geq 0$, ${\rm hc}\left( (\#^{r}S^2\times S^1)\#M(E^{k}_{0})\right) =\lceil \frac{r}{2} \rceil+1$ if $k=1$ or $k\geq 3$.
\item ${\rm hc}(M(E^{2}_{0}))=2$, and for $r\geq 1$ ${\rm hc}\left( (\#^{r}S^2\times S^1)\#M(E^{2}_{0})\right) =\lceil \frac{r}{2} \rceil+1$.
\item For $r\geq 0$, ${\rm hc}\left( (\#^{r}S^2\times S^1)\#M(E^{2}_{1})\right) =\lceil \frac{r}{2} \rceil+1$.
\item For $r\geq 0$, ${\rm hc}\left( (\#^{r}S^2\times S^1)\#M(E^{k}_{1})\right) =\lfloor \frac{r}{2} \rfloor+2$ if $k\geq 3$.
\end{itemize}
\end{prop} 

As a preparation, we review two observations: 
Firstly, the invariant ${\rm hc}(\cdot)$ is subadditive under the connected sum i.e. ${\rm hc}(M_{1}\#M_{2})\leq {\rm hc}(M_{1})+{\rm hc}(M_{2})$ since we get a homological fiber by the plumbing of two homological fibers, see \cite{sekino} for example. 
Secondly, it is known that $S^2\times S^1$ has a fibered link whose fiber surface is an annulus. 
By the plumbing it and a Hopf annulus in $S^3$, and by the plumbing two fibered annuli of $S^2\times S^1$, we get genus one fibered knots in $S^2\times S^1$ and $\#^{2}S^2\times S^1$. 
Since they are not integral homology $3$-spheres, we see that ${\rm hc}(S^2\times S^1)=1$ and ${\rm hc}\left( \#^{2}S^2 \times S^1\right) =1$.

\subsection{ ${\rm hc}\left( (\#^{r}S^2\times S^1)\#M(E^{k}_{0})\right) $}\label{hcek0}
We divide the argument into three cases, where $r$ is zero, where $r$ is positive even, and where $r$ is odd.
\subsubsection{${\rm hc}\left( M(E^{k}_{0})\right) $}\label{hcek0_r0}
We will compute ${\rm hc} \left( M(E^{k}_0)\right)$. 
Since $M(E^{k}_0)$ is not an integral homology 3-sphere, ${\rm hc}\left( M(E^{k}_0)\right) \geq 1$. 
Moreover, it is known that $M(E^{k}_0)$ has a fibered link whose fiber surface is homeomorphic to $\Sigma_{0,3}$ as in Figure~\ref{e^k_0_fib}: 
If we ignore the green curve, the right of Figure~\ref{e^k_0_fib} represents a fiber surface of a fibered link whose fiber surface is homeomorphic to $\Sigma_{0,3}$ in $L(2^{k},1)\#L(2^{k},1)$, which is obtained by the plumbing of fibered annuli in each prime components in appropriate way. 
The green curve is on this fiber, and note that the canonical framing of this curve is identical with the surface framing. 
Then $\left (-\frac{1}{2^{k}}\right)$-surgery (with respect to the canonical framing) along the green curve preserves the fiber structure. 
By the plumbings of two Hopf annuli, we get a genus two fibered knot in $M(E^{k}_0)$. 
This implies ${\rm hc}\left( M(E^{k}_0)\right) \leq2$. 
Thus it is enough to show whether $M(E^{k}_{0})$ has a genus one homologically fibered knot or not.\\

\begin{figure}[htbp]
 \begin{center}
  \includegraphics[width=150mm]{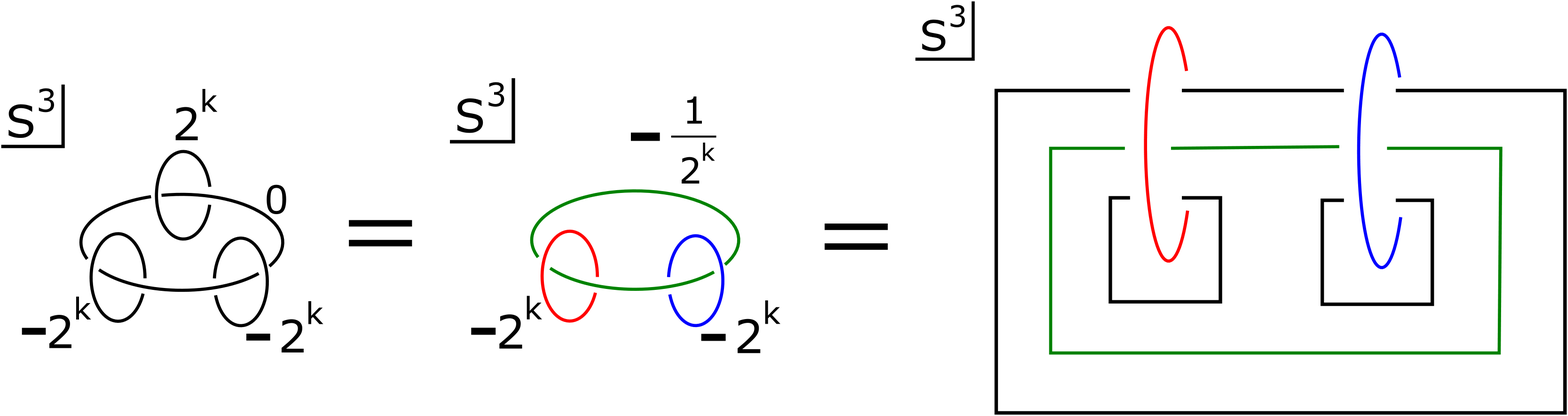}
 \end{center}
 \caption{Surgery link on a planar fiber surface}
 \label{e^k_0_fib}
\end{figure}

By Theorem~\ref{thm}, $M(E^{k}_0)$ has a genus one homoloically fibered knot if and only if there exist integers $x,y,z,w,\alpha,\beta$ and $\gamma$ satisfying 
\begin{eqnarray} \label{hc_mek0}
   \pm1 &=&
  \left| \begin{array}{cccc}
     0    &  2^{k} &  x  &  y  \\
     2^{k}    &  0  &   z  &  w \\
    x  &  z   &  \alpha  &  \beta +1\\
    y  &  w  &  \beta   &   \gamma \\
  \end{array} \right|   \nonumber \\
          &=& -2^{2k}\{ \alpha \gamma - \beta(\beta +1)\}
 +2^{k+1}\{x(z\gamma -w\beta)-y(z\beta -w\alpha)\}-2^{k}(xw+yz)+(xw-yz)^{2}
\end{eqnarray}

\vspace{0.5cm}
(i) $k=1$\\
In this case, it is known that $M(E^{1}_{0})$ has a genus one fibered knot as in Figure~\ref{e^1_0_fib}: 
If we ignore the green curve, the right of Figure~\ref{e^1_0_fib} represents a fiber surface of a genus one fibered knot in $L(2,1)\#L(2,1)$, which is obtained by the plumbing of fibered annuli in each prime components in appropriate way. 
The green curve is on this fiber, and note that the surface framing of this curve is the $(-1)$-slope with the canonical framing. 
This implies that the $\left (-\frac{1}{2}\right)$-slope of the green curve with respect to the canonical framing is the $\left( -\frac{1}{2}\right)$-slope with respect to the surface framing. 
Then the surgery along the green curve preserves the fiber structure. 
Thus $M(E^{1}_{0})$ has a genus one homologically fibered knot. 
In fact, $x=3,y=1,z=0,w=1,\alpha=-1,\beta=0$ and $\gamma=0$ is one of the solutions. 
This solution may correspond to a non-fibered knot. 

\begin{figure}[htbp]
 \begin{center}
  \includegraphics[width=100mm]{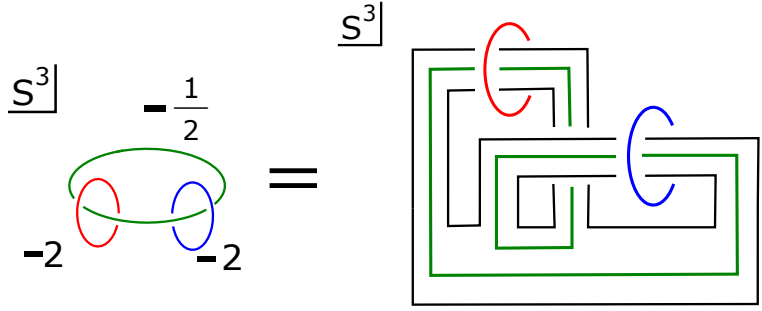}
 \end{center}
 \caption{Surgery link on a fiber surface of genus one}
 \label{e^1_0_fib}
\end{figure}

(ii) $k=2$\\
In this case, there exist no solutions. 
If there was, $(xw-yz)$ must be odd since the other terms in the right hand side of (\ref{hc_mek0}) are even. 
If $(xw-yz)$ is odd, then the right hand side of (\ref{hc_mek0}) is congruent to $5$ modulo $8$ by noting that $xw+yz=xw-yz+2yz$ and that the square of every odd integer is congruent to $1$ modulo $8$. 
This cannot be $\pm1$. 

(iii) $k\geq3$\\
In this case, $x=2^{k-1}, y=1, z=1, w=1, \alpha=2^{k-3}+1, \beta=0$ and $\gamma=0$ is one of the solutions of (\ref{hc_mek0}).

\subsubsection{${\rm hc}\left( (\#^{2m}S^2 \times S^1)\# M(E^{k}_{0})\right) $ for $m\geq 1$}\label{hcek0_evenr}
Since $H_{1}\left( (\#^{2m}S^2 \times S^1)\# M(E^{k}_{0}) \right) \cong \mathbb{Z}^{2m}\oplus \mathbb{Z}/2^{k}\mathbb{Z}\oplus \mathbb{Z}/2^{k}\mathbb{Z}$ requires at least $2m+2$ elements for generating, ${\rm hc}\left( (\#^{2m}S^2 \times S^1)\# M(E^{k}_{0})\right) \geq m+1$ by Remark~\ref{lowerbound}. 
By the plumbing of two fibered annuli of $S^2 \times S^1$ to a fiber surface of $M(E^{k}_{0})$ in the right of Figure~\ref{e^k_0_fib}, we get a genus two fibered knot in $(\#^{2}S^2 \times S^1)\# M(E^{k}_{0})$. 
Thus 
\begin{eqnarray*} \label{hc_mek0}
    {\rm hc}\left( (\#^{2m}S^2 \times S^1)\# M(E^{k}_{0})\right)  &\leq&
 (m-1){\rm hc}( \#^{2}S^2 \times S^1) + {\rm hc}\left( (\#^{2}S^2 \times S^1)\# M(E^{k}_{0})\right) \\
          &\leq& 
 m+1.
\end{eqnarray*}
Therefore, ${\rm hc}\left( (\#^{2m}S^2 \times S^1)\# M(E^{k}_{0})\right) =m+1$ for $m\geq 1$.

\subsubsection{${\rm hc}\left( (\#^{2m+1}S^2 \times S^1)\# M(E^{k}_{0})\right) $ for $m\geq 0$} \label{hcek0_oddr}
Since $H_{1}\left( \left( (\#^{2m+1}S^2 \times S^1)\# M(E^{k}_{0})\right) \right) \cong \mathbb{Z}^{2m+1}\oplus \mathbb{Z}/2^{k}\mathbb{Z}\oplus \mathbb{Z}/2^{k}\mathbb{Z}$ requires at least $2m+3$ elements for generating, ${\rm hc}\left( (\#^{2m+1}S^2 \times S^1)\# M(E^{k}_{0})\right) \geq m+2$ by Remark~\ref{lowerbound}. 
By the plumbing of two fibered annuli, one is in $S^2 \times S^1$ and the other is in $S^3$, to a fiber surface of $M(E^{k}_{0})$ in the right of Figure~\ref{e^k_0_fib}, we get a genus two fibered knot in $(S^2 \times S^1)\# M(E^{k}_{0})$. 
Thus 
\begin{eqnarray*} \label{hc_mek0}
   {\rm hc}\left( (\#^{2m+1}S^2 \times S^1)\# M(E^{k}_{0})\right)  &\leq&
  m\cdot {\rm hc}( \#^{2}S^2 \times S^1) +{\rm hc}\left( (S^2 \times S^1)\# M(E^{k}_{0})\right)\\
          &\leq& 
 m+2.
\end{eqnarray*}
Therefore, ${\rm hc}\left( (\#^{2m+1}S^2 \times S^1)\# M(E^{k}_{0})\right) =m+2$ for $m\geq 0$.

\subsection{${\rm hc}\left( (\#^{r}S^2\times S^1)\#M(E^{k}_{1})\right)$}\label{hcek1}
We divide the argument into four cases, where $r$ is zero, where $k$ is two, where $k$ is not two and $r$ is even, and where $k$ is not two and $r$ is odd.

\subsubsection{${\rm hc}\left( M(E^{k}_{1})\right) $}
We will compute ${\rm hc}(M(E^{k}_{1}))$ for $k\geq2$. 
Since $M(E^{k}_1)$ is not an integral homology 3-sphere, ${\rm hc}\left( M(E^{k}_{1})\right) \geq 1$. 
Moreover, $M(E^{k}_1)$ has a genus two fibered knot as in Figure~\ref{e^k_1_fib}: 
If we ignore the red curve and the blue curve, two of the bottom of Figure~\ref{e^k_1_fib} represent fiber surfaces of genus two fibered knots in $L(2^{k},1)\#L(2^{k},1)\#L\left( -\frac{2}{3} \{2^{k-1}-(-1)^{k-1}\}, 1\right)\#L\left(2,(-1)^{k-1}\right)$ for odd $k$ and even $k$, which are obtained by the plumbings of fibered annuli in each prime components in appropriate way. 
The red curve and blue curve are on this fiber. 
Note that the $0$-slope of the red curve with respect to the canonical framing is the $(-1)$-slope with respect to the surface framing, and that the $\left((-1)^{k}\cdot2\right)$-slope of the blue curve with respect to the canonical framing is the $1$-slope with respect to the surface framing. 
Then the surgery along the red curve and the blue curve preserves the fiber structure. 
Therefore we have ${\rm hc}\left( M(E^{k}_1)\right) \leq 2$. 
Thus it is enough to show whether $M(E^{k}_{1})$ has a genus one homologically fibered knot or not.\\

\begin{figure}[htbp]
 \begin{center}
  \includegraphics[width=120mm]{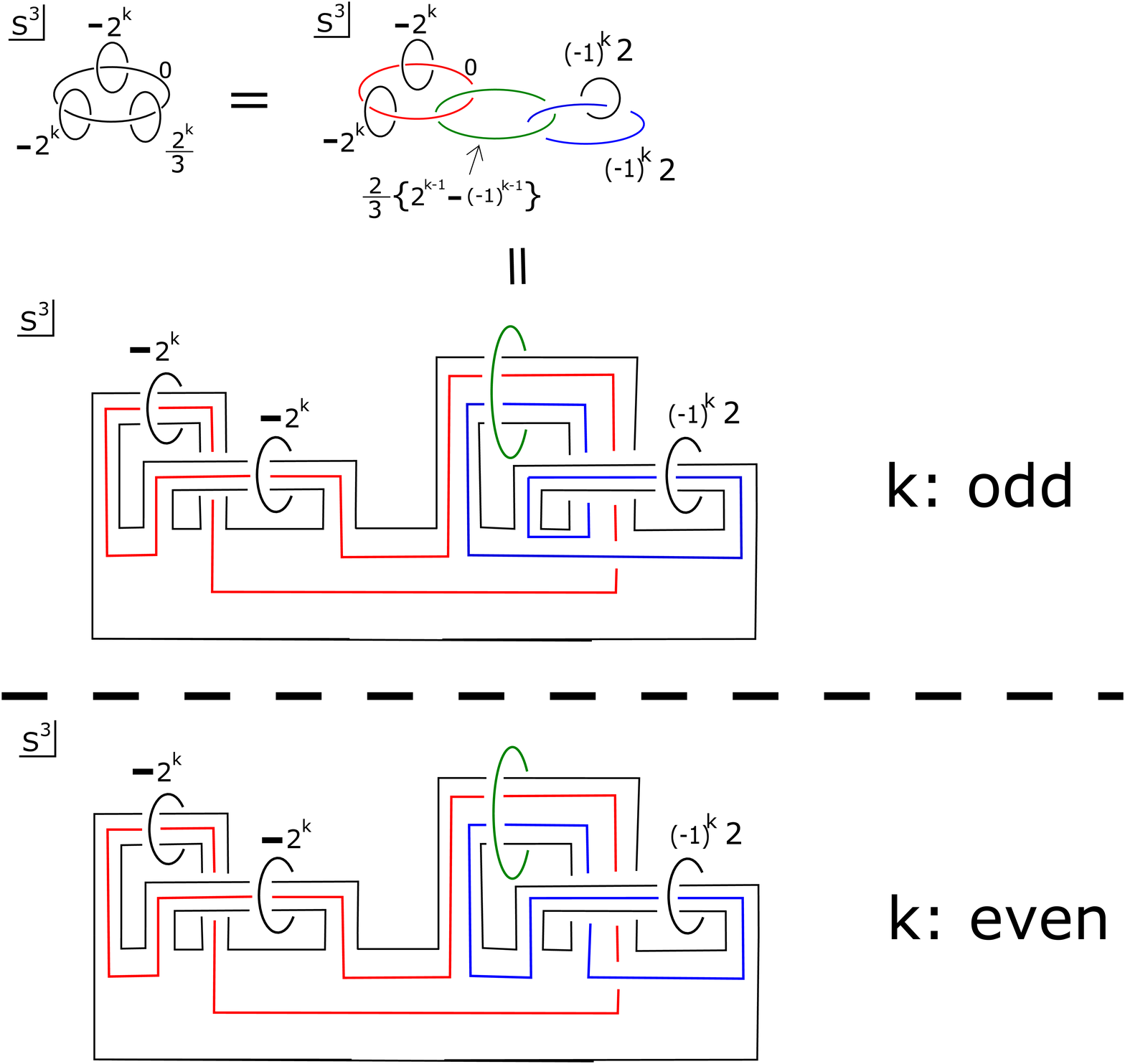}
 \end{center}
 \caption{Surgery link on a fiber surface of genus two}
 \label{e^k_1_fib}
\end{figure}

By Theorem~\ref{thm}, $M(E^{k}_1)$ has a genus one homologically fibered knot if and only if there exist integers $x,y,z,w,\alpha,\beta$ and $\gamma$ satisfying 
\begin{eqnarray} \label{hc_mek1}
   \pm1 &=&
  \left| \begin{array}{cccc}
     2^{k+1}    &  -2^{k} &  -x  &  -y  \\
     -3\cdot2^{k}    &  2^{k+1}  &   -3z  &  -3w \\
    x  &  z   &  \alpha  &  \beta +1\\
    y  &  w  &  \beta   &   \gamma \\
  \end{array} \right|   \nonumber \\
       &=& 2^{2k}\{ \alpha \gamma - \beta(\beta +1)\}
 +2^{k+1}\{\gamma(x^{2}+3xz+3z^{2})+\alpha(y^{2}+3yw+3w^{2})-(2\beta+1)(xy+3zw)\} \nonumber\\   
 &\ & \hspace{2.0cm} -3\cdot2^{k}(2\beta+1)(xw+yz)+3(xw-yz)^{2}
\end{eqnarray}

\vspace{0.5cm}
(i) $k=2$\\
In this case, we get a genus one fibered knot as in Figure~\ref{e^2_1_fib}: 
If we ignore the green curve, the right of Figure~\ref{e^2_1_fib} represents a fiber surface of a genus one fibered knot in $L(4,1)\#L(4,1)$, which is obtained by the plumbing of fibered annuli in each prime components in appropriate way. 
The green curve is on this fiber, and note that the surface framing of this curve is the $(-1)$-slope with the canonical framing. 
This implies that the $\left (-\frac{3}{4}\right)$-slope of the green curve with respect to the canonical framing is the $\left( {1}{4}\right)$-slope with respect to the surface framing. 
Then the surgery along the green curve preserves the fiber structure. 
Thus $M(E^{2}_{1})$ has a genus one homologically fibered knot. 
In Fact, $x=-1,y=1,z=1,w=0,\alpha=1,\beta=0$ and $\gamma=0$ is one of the solutions. 
This solution may correspond to a non-fibered knot.
Thus we have ${\rm hc}\left( M(E^{2}_{1})\right) =1$. \\

\begin{figure}[htbp]
 \begin{center}
  \includegraphics[width=150mm]{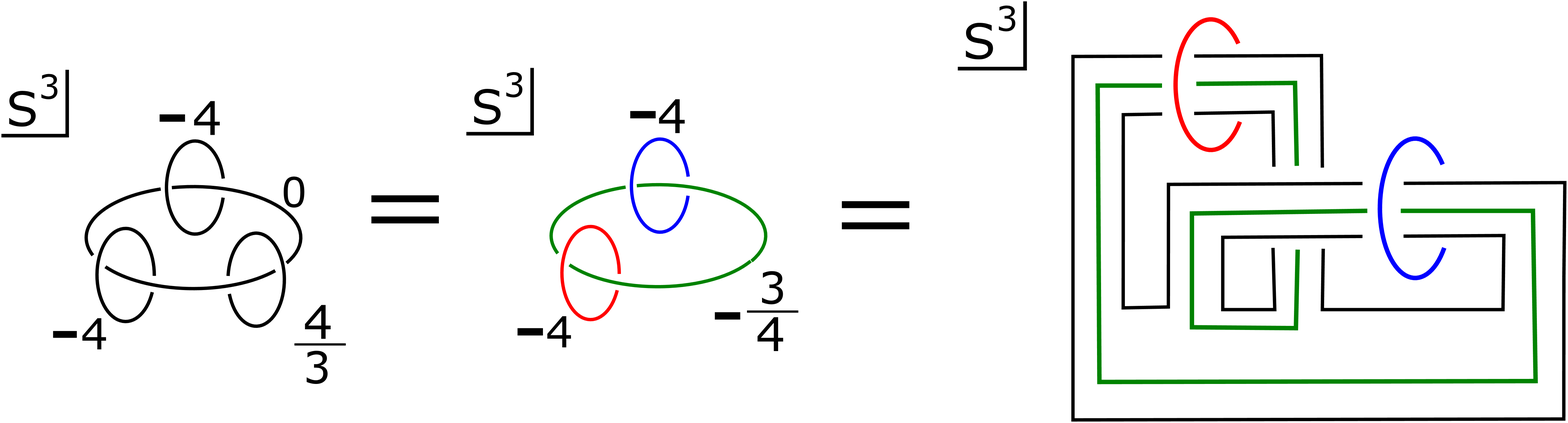}
 \end{center}
 \caption{Surgery link on a fiber surface of genus one}
 \label{e^2_1_fib}
\end{figure}

(ii) $k\geq3$\\
In this case, there exist no solutions of (\ref{hc_mek1}). 
In fact, all terms but the last in the right hand side of (\ref{hc_mek1}) are divisible by $8$ and the last term is congruent to neither $1$ nor $-1$ modulo $8$. 
Thus we have ${\rm hc}\left( M(E^{k}_{1})\right) \geq2$. 
We conclude ${\rm hc}\left( M(E^{k}_{1})\right) =2$. 

\subsubsection{${\rm hc}\left( (\#^{r}S^2\times S^1)\#M(E^{2}_{1})\right) $ for $r\geq 0$}
Since $H_{1}\left( \left( (\#^{r}S^2 \times S^1)\# M(E^{2}_{1})\right) \right) \cong \mathbb{Z}^{r}\oplus \mathbb{Z}/4\mathbb{Z}\oplus \mathbb{Z}/4\mathbb{Z}$ requires at least $r+2$ elements for generating, 
${\rm hc}\left( (\#^{r}S^2 \times S^1)\# M(E^{2}_{1})\right) \geq \lceil \frac{r}{2} \rceil +1$ by Remark~\ref{lowerbound}. 
On the other hand, 
\begin{eqnarray*} \label{hc_mek0}
   {\rm hc}\left( (\#^{r}S^2 \times S^1)\# M(E^{2}_{1})\right)  &\leq&
   {\rm hc}( \#^{r}S^2 \times S^1) + {\rm hc}\left( M(E^{2}_{1})\right) \\  
          &\leq& 
 \lceil \frac{r}{2} \rceil +1.
\end{eqnarray*}
Therefore, ${\rm hc}\left( (\#^{r}S^2 \times S^1)\# M(E^{2}_{1})\right) =\lceil \frac{r}{2} \rceil +1$ for $r\geq 0$.

\subsubsection{${\rm hc}\left( (\#^{2m}S^2\times S^1)\#M(E^{k}_{1})\right) $ for $m\geq 0$ and $k\geq 3$}
Note that 
\begin{eqnarray*} \label{hc_mek0}
  {\rm hc}\left( (\#^{2m}S^2\times S^1)\#M(E^{k}_{1})\right)  &\leq&
  m\cdot {\rm hc}(\#^{2}S^2\times S^1)+  {\rm hc}\left( M(E^{k}_{1})\right) \\  
          &=& 
 m+2.
\end{eqnarray*}
We will show that $(\#^{2m}S^2\times S^1)\#M(E^{k}_{1})$ has no homologically fibered links whose homological fibers are homeomorphic to $\Sigma_{m+1,1}$. 
This implies that ${\rm hc}\left( (\#^{2m}S^2\times S^1)\#M(E^{k}_{1})\right)=m+2$.\\
Suppose that $(\#^{2m}S^2\times S^1)\#M(E^{k}_{1})$ had a homologically fibered links whose homological fibers are homeomorphic to $\Sigma_{m+1,1}$. 
By applying Theorem~\ref{thm}, we have a solution $X$, $(2m+2) \times (2m+2)$-integer matrix, and $Y$, $(2m+2) \times (2m+2)$-symmetric integer matrix, for an equation below.

\begin{eqnarray} \label{abcdef}
   \left| \begin{array}{cc}
     O_{2m} \oplus   
                           \left( \begin{array}{cc}
                             2^{k+1}     &  -2^{k}  \\
                             -3\cdot 2^{k}    &  2^{k+1} \\
                           \end{array}  \right)

  & \left( I_{2m} \oplus \left( \begin{array}{cc}
                        -1    &  0  \\
                         0    &  -3 \\
                        \end{array} \right) \right)X  \\
      X^{t}    &  Y+\mathcal{E}\\
  \end{array} \right| = \pm 1
\end{eqnarray}

Note that the left hand side is congruent to $\left| X^{t}\left( I_{2m} \oplus \left( \begin{array}{cc}
                        -1    &  0  \\
                         0    &  -3 \\
                        \end{array} \right) \right) X \right| =3 \left| X\right| ^{2}$ modulo $8$. 
This cannot be $\pm1$, and this leads a contradiction.

\subsubsection{${\rm hc}\left( (\#^{2m+1}S^2\times S^1)\#M(E^{k}_{1})\right) $ for $m\geq 0$ and $k\geq 3$}
First, we show that $M(E^{k}_{1})$ has a homologically fibered link whose homological fiber is homeomorphic to $\Sigma_{1,2}$. 
By Theorem~\ref{thm}, the existence of such a link is equivalent to that of a solution $a,b,c,d,e,f,v,w,x,y,z$ of the equation below. 

\begin{eqnarray} \label{qwerty}
   \left| \begin{array}{ccccc}
     2^{k+1} & -2^{k} & -x & -y & -u \\
     -3\cdot 2^{k} & 2^{k+1}   &  -3z  &  -3w  &  -3v \\
    x  &  z  &  a  &  b+1  & d  \\
    y  & w  &  b  &  c  &  e  \\
    u  & v  &  d   &  e  &  f\\
  \end{array} \right| = \pm 1
\end{eqnarray}

By substituting $u=1, v=d=e=f=0$, the left hand side of (\ref{qwerty}) becomes equal to 
$\left| \begin{array}{ccc}

 2^{k+1}   &  -3z  &  -3w  \\
   z  &  a  &  b+1   \\
    w  &  b  &  c   \\
  \end{array} \right|$. 
Note that the existence of $a,b,c,z,w$ such that the determinant becomes $\pm 1$ is equivalent to that of a genus one homologically fibered knot in $L(2^{k+1},3)$ by Theorem~\ref{thm}, and the existence is guaranteed by Nozaki. 
Thus $M(E^{k}_{1})$ has a homologically fibered link whose homological fiber is homeomorphic to $\Sigma_{1,2}$. 
Moreover, we get a homologically fibered link whose homological fiber is homeomorphic to $\Sigma_{2,1}$ in $(S^2\times S^1)\#M(E^{k}_{1})$ by the plumbing with a fibered annulus in $S^2\times S^1$. \\

Since $H_{1}\left( \left( (\#^{2m+1}S^2 \times S^1)\# M(E^{k}_{1})\right) \right) \cong \mathbb{Z}^{2m+1}\oplus \mathbb{Z}/2^{k}\mathbb{Z}\oplus \mathbb{Z}/2^{k}\mathbb{Z}$ requires at least $2m+3$ elements for generating, ${\rm hc}\left( (\#^{2m+1}S^2 \times S^1)\# M(E^{k}_{1})\right) \geq m+2$ by Remark \ref{lowerbound}.  
On the other hand, 
\begin{eqnarray*} \label{hc_mek0}
  {\rm hc}\left( (\#^{2m+1}S^2 \times S^1)\# M(E^{k}_{1})\right) &\leq&
  m\cdot {\rm hc}( \#^{2}S^2 \times S^1)+ {\rm hc}\left( (S^2 \times S^1)\# M(E^{k}_{1})\right)\\  
          &\leq& 
m+2.
\end{eqnarray*}
Therefore, ${\rm hc}\left( (\#^{2m+1}S^2 \times S^1)\# M(E^{k}_{1})\right) =m+2$ for $m\geq 0$.

\vspace{0.5cm}

\ GRADUATE SCHOOL OF MATHEMATICAL SCIENCES, THE UNIVERSITY OF TOKYO, 3-8-1 KOMABA, MEGURO--KU, TOKYO, 153-8914, JAPAN\\
\ \ E-mail address: \texttt{sekinonozomu@g.ecc.u-tokyo.ac.jp}
\end{document}